\documentclass[10pt,english,letter,reqno]{amsart}

\usepackage{amssymb,url,xspace,amsthm}
\usepackage{amsmath}
\usepackage{mathtools}
\usepackage[T1]{fontenc}
\usepackage{graphicx}
\usepackage{enumitem}
\usepackage{amsrefs}

\usepackage{hyperref}
\hypersetup{
  colorlinks = true,
  urlcolor   = blue,
  linkcolor  = blue,
  citecolor  = blue
}

\newcommand{\Ostar}{O^{\ast}}

\theoremstyle{plain}
\newtheorem{theorem}{Theorem}
\newtheorem{proposition}[theorem]{Proposition}
\newtheorem{lemma}[theorem]{Lemma}
\newtheorem{corollary}[theorem]{Corollary}
\theoremstyle{definition}
\newtheorem{remark}[theorem]{Remark}

\author[A. Fiori]{Andrew Fiori}
\address{Department of Mathematics and Statistics, University of Lethbridge,
4401 University Drive,
Lethbridge, Alberta,
T1K 3M4,
Canada}
\email{andrew.fiori@uleth.ca}
\thanks{Andrew Fiori thanks and acknowledges the University of Lethbridge for their financial support as well as the support of NSERC Discovery Grant RGPIN-2020-05316.}


\title{Zero Density Theorems for Short Intervals}

\begin{document}
\begin{abstract}
We provide explicit bounds on the number of zeros for $\zeta$ in short intervals near the edge of the critical strip.
 Ultimately we show that even in short intervals a positive proportion of zeros will not be near the edge of the critical strip.
 For example with $h>100$ and $t>10^{100}$ the proportion of zeros of $\zeta$ with imaginary part in $[t-h,t+h]$ which have real part in $[\frac{1}{16},\frac{15}{16}]$ is at least $13\%$ and a positive proportion is already obtained at $t>10^{43}$.
\end{abstract}

\maketitle

The aim of this article is to provide bounds on the number of zeros in small sections of the critical strip, particularly along the boundary of the critical strip. We will look at primarily two types of regions, circular, and rectangular.

Throughout the article we shall denote by
\begin{itemize}
    \item $N(T)$ the number of zeros, $\rho$ of $\zeta$ with imaginary $0 < \Im(\rho) \leq T$.
    \item $N(s,\alpha)$ the number of zeros $\rho$ of $\zeta$ with $|\rho-s|\leq \alpha$.
    \item $N(T_1,T_2,\alpha)$ the number of zeros $\rho$ of $\zeta$ with $T_1 \leq \Im(\rho) \leq T_2$
    and $\Re(\rho)>1-\alpha$.
\end{itemize}
In all cases zeros are counted with multiplicity and we shall maintain the convention that zeros on the boundary count with $1/2$ their multiplicity. We note that using this last convention, as opposed to including or excluding the boundary does not impact the bounds we would prove except for the degenerate cases $T=0$, $\alpha=0$, or $T_1=T_2$.

Note that other sources write $N(\sigma,T)$ for  the number of zeros $\rho$ of $\zeta$ with $0 < \Im(\rho) \leq T$ and $\Re(\rho)>\sigma$. We don't anticipate this causing any confusion.

There are many results on zero density, however most of these consider $T_2-T_1$ to be large in comparison to the size of $T_2$. Our aim is to treat much smaller intervals. At the same time, we note that bounds on the number of zeros in short intervals are a key ingredient to a common strategy for proving bounds on zeros in long intervals (see for example \cite{Bellotti2025}). 

One of the simplest to interpret corollaries to our main theorem is
\begin{corollary}\label{cor:simpmain}
Suppose $0<\alpha<1/6$. For $h_0$ and $t_0$ sufficiently large, and for all $t>t_0$ and $t^{2/3}>h>h_0$,
the proportion of zeros $\rho$ of $\zeta(s)$ with $\Im(\rho) \in [t-h,t+h]$ for which $\Re(\rho) \in [\alpha,1-\alpha]$
is bounded below by an explicitly computable positive constant.
That is
\[  1-\frac{2N(t-h,t+h,\alpha)}{N(t+h)-N(t-h)} >  C_{\alpha,h_0,t_0} > 0. \]
Example values of $C_{\alpha,h_0,t_0}$ are given in Table~\ref{table:chat}.
\end{corollary}
This is an immediate consequence of Corollary~\ref{cor:main-jensen} or Corollary~\ref{cor:main-littlewood}, taking $C_{\alpha,h_0,t_0} = 1-2C_2$ for the relevant $C_2$; the factor of $2$ accounts for the zeros with $\Re(\rho)<\alpha$, which the functional equation puts in bijection with those having $\Re(\rho)>1-\alpha$.

\begin{table}[h]
\caption{Table of values for Corollary~\ref{cor:simpmain}. For $t>t_0$ and $t^{2/3}>h>h_0$ the proportion of zeros with imaginary part within $h$ of $t$ whose real part is in $[\alpha,1-\alpha]$ is
at least $C_{\alpha,h_0,t_0}$. The column ``mechanism'' records which of Corollary~\ref{cor:main-jensen} (J) and Corollary~\ref{cor:main-littlewood} (L) gives the stronger bound. The constants $C_{\alpha,h_0,t_0}$ are rounded down.}\label{table:chat}
\begin{tabular}{|ccc|cc|}
\hline
$\alpha$&$h_0$&$t_0$& mechanism & $C_{\alpha,h_0,t_0}$ \\
\hline
$1/7$&$100$&$10^{10000}$&L&$0.0628$\\
$1/8$&$100$&$10^{1000}$&L&$0.0970$\\
$1/9$&$10$&$10^{10000}$&L&$0.0760$\\
$1/10$&$5$&$10^{10000}$&L&$0.0236$\\
$1/12$&$5$&$10^{10000}$&L&$0.1040$\\
$1/16$&$5$&$10^{10000}$&L&$0.2122$\\
$1/16$&$100$&$10^{43}$&L&$0.0054$\\
$1/16$&$100$&$10^{100}$&L&$0.1374$\\
$1/16$&$100$&$10^{10000}$&L&$0.3484$\\
$1/32$&$100$&$10^{100}$&L&$0.2157$\\
$1/64$&$2$&$10^{1000}$&J&$0.1649$\\
$1/64$&$100$&$10^{50}$&L&$0.1312$\\
$1/128$&$100$&$10^{1000}$&L&$0.5527$\\
$1/256$&$1$&$10^{10000}$&L&$0.0934$\\
$1/2048$&$100$&$10^{10000}$&L&$0.7163$\\
\hline
\end{tabular}
\end{table}
The above says that a positive proportion of zeros are not near the edge of the critical strip in relatively small intervals.
Tables~\ref{tab:C2J} and \ref{tab:C2L} allow one to get many more explicit examples of positive proportions, whereas Tables~\ref{tab:C2Jfrontier} and \ref{tab:C2Lfrontier} illustrate minimal values for $t_0$ for which a positive proportion result is attainable.

This result is in some ways qualitatively weaker than other estimates about zeros being on or near the critical line. 
For example the approach of Levinson, as improved by Conrey and then Pratt, Robles, Zaharescu and Zeindler (See \cite{Levinson1974,Conrey1989,PrattEtAl2020})
establishes that, as $t\rightarrow \infty$, a little over 40\% are on the critical line. However, the approach only appears to work on long intervals, that is, intervals of length at least $\omega(T)$.
In fact, since roughly $40\%$ of the zeros less than $(5/4)T$ are less than $(3/4)T$ the result barely tells you that in an interval of width $T/2$ at height $T$ there must be any zeros on the critical line. The recent claimed improvement to $2/3$ in \cite{AlpogeFurman2026} doesn't fundamentally alter this.
Karatsuba's proof of Selberg's conjecture  does lower that threshold to powers of $T$, and subsequent work ultimately establishes that almost every relatively short (sub-exponential in $\log(T)$) interval contains zeros on the critical line (see \cite{Karatsuba1993}). But almost all intervals still isn't all, and they still aren't obtaining a positive proportion.
Standard zero-density results going back at least to Bohr and Landau (see \cite{BohrLandau1913}) tell us 100\% of zeros are within $\epsilon$ of the critical line. But again this tells us nothing about short intervals. 
In the opposite direction, results of Montgomery in fact tell us that if there are zeros very near the edge of the critical strip, there will somehow be many of them (see \cite[Chapter~11]{Montgomery1973}).

We recall that \cite{PlattTrudgian2021} has verified that up to height $H_0 := 3\cdot 10^{12}$ every non-trivial zero of $\zeta$ has been
verified to lie exactly on the critical line.

The above illustrates the result; the new type of theorem we are actually providing is:
\begin{theorem}[Jensen mechanism]\label{thm:main-jensen}
Suppose $0<\alpha<r<1$, $\alpha<1/6$. For any $t\geq e$ and $t^{2/3}\geq h\geq 0$ we have
\[ N(t-h,t+h,\alpha) \;<\; U_J(\alpha,r,h,t) + \Ostar\!\left(\frac{3.35}{D_{r,\alpha}\,t^{1/3}}\right).\]
where
\begin{align*} 
U_J(\alpha,r,h,t) := &(2h+2\hat{\alpha}_r)\!\left(\frac{B_{1,\alpha,r}}{\pi}\log|t| + B_{2,\alpha,r}\log\log|t| + B_{3,\alpha,r}\right) + B_{4,\alpha,r}
\end{align*}
with $\hat{\alpha}_r$, $B_{1,\alpha,r},\ldots,B_{4,\alpha,r}$ as in Theorem
\ref{thm:rectangularjensen} below. Sample values are in Table~\ref{tab:rectangularjensen}.
\end{theorem}
\begin{proof}
The case of $t-h>H_0$ is precisely a consequence of Theorem~\ref{thm:rectangularjensen}, restricted to  $h<t^{2/3}$ and $t>10^{12}$.

The $t+h<H_0$ case is truly trivial.

For the case $t-h<H_0$ and $t+h>H_0$ simply replace $t$ and $h$ by $t' = t + \delta$ and $h' = h-\delta$ where $\delta = \frac{1}{2}(H_0-t+h)$
 so that $t'-h'=H_0$ and $t'+h'=t+h$. Then Theorem~\ref{thm:rectangularjensen}, applied with $t'$ and $h'$, translates back to a stronger claim in terms of $t$ and $h$
 than what the Theorem claims.
\end{proof}
Note that if $h\gg t^{2/3}$ and $t>10^{12}$, the claim becomes an easy consequence of standard zero-density estimates for $\zeta$. We provide a sketch but omit the details.
By \cite{KadiriLumleyNg2018} we have 
\[ N(t-h,t+h,\alpha) \ll (t+h)^{\frac{8}{3}\alpha}\log(t+h)^3 \ll   U_J(\alpha,r,h,10^{12}) < U_J(\alpha,r,h,t)  \]
Since $\alpha<1/6$ and $h>10^8$ the coefficient $\frac{\log(h)^3}{h^{1-4\alpha}}$ shrinks faster as $\alpha\to 0$ than $B_{1,\alpha,r}$ can.
The main obstacle between $t^{2/3} < h \ll t^{2/3}$ is the $\Ostar$ term, which would need to be weakened in the statement above, or tightened in our source.

\begin{theorem}[Littlewood mechanism]\label{thm:main-littlewood}
Suppose $0<\alpha<r$, $\alpha<1/6$. For any $t\geq e^e$ and $t^{2/3}\geq h\geq 0$ we have
\[  N(t-h,t+h,\alpha) \;<\;  U_L(\alpha,r,h,t) + \Ostar\!\left(\frac{1}{t(r-\alpha)}\left(0.335 + \frac{0.51}{\log|t|}\right)\right), \]
where
\begin{align*}
 U_L(\alpha,r,h,t) :=
&\frac{2A_{1,\alpha,r}h+2A_{4,\alpha,r}}{\pi}\log|t| + (2A_{2,\alpha,r}h+2A_{5,\alpha,r})\log\log|t|
\\&+ 2A_{3,\alpha,r}h+A_{6,\alpha,r}
\end{align*}
with $A_{1,\alpha,r},\ldots,A_{6,\alpha,r}$ as in Theorem~\ref{thm:littlewoodshort} below. Sample values are in Table~\ref{tab:littlewoodshort}.
\end{theorem}
\begin{proof}
The proof is as in Theorem
\ref{thm:main-jensen}'s proof, using Theorem
\ref{thm:littlewoodshort} in place of Theorem~\ref{thm:rectangularjensen}
\end{proof}

Using that by \cite[Corollary~1.3]{BellottiWong2026} we have
\[ \left|N(t) - \frac{t}{2\pi}\log\frac t{2\pi e}\right| <  0.097\log|t|+4.954 \]
we conclude that, for $t>10^{12}$ and $0<h<t-1$ we have:
\[ N(t+h) - N(t-h) \geq L(h,t) := \left(\frac{1}{\pi}\log\frac{t}{2\pi}-\frac{(1-\log 2)h^2}{\pi t^2} \right)h - 0.194\log|t| - 9.908. \]
This directly leads to the following two Corollaries, one per mechanism.
\begin{corollary}[Jensen]\label{cor:main-jensen}
Suppose $0<\alpha<r<1$ and $\alpha<1/6$ and then for $t_0>10^{12}$ and
\[ h_0>\frac{97\pi}{500}\left(\frac{1 + \frac{4954}{97\log(t_0)}}{1 -\frac{1-\log 2}{t_0^{2/3}\log(t_0)}-\frac{\log(2\pi)}{\log(t_0)}}\right) \]
 we have that for all $t>t_0$ and $t^{2/3}>h>h_0$
\[\frac{ N(t-h,t+h,\alpha)}{N(t+h)-N(t-h)}  < C_2^J(\alpha,r,h_0,t_0) \]
where
\begin{align*}  C_2^J(\alpha,r,h,t)   :=& \frac{(2+\frac{2\hat{\alpha}_r}{h})\!\left(B_{1,\alpha,r} + \frac{\pi B_{2,\alpha,r}\log\log|t|}{\log|t|} + \frac{\pi B_{3,\alpha,r}}{\log|t|}\right) + \frac{\pi B_{4,\alpha,r}}{h\log|t|} + \frac{3.35\pi}{D_{r,\alpha}h t^{1/3}\log|t|} }
{1 - \frac{\log(2\pi)}{\log|t|} -  \frac{1-\log 2}{t^{2/3} \log|t|} - \frac{0.194\pi}{h} - \frac{9.908\pi}{h\log|t|}}
\end{align*}
is decreasing in both $h_0$ and $t_0$ for any fixed $\alpha$ and $r$.
A table of values for $C_2^J(\alpha,r,h_0,t_0)   $ is given in Table~\ref{tab:C2J}.
\end{corollary}
\begin{proof}
By inspection we have
\[ \frac{U_J(\alpha,r,h,t)}{L(h,t)} < \frac{(2+\frac{2\hat{\alpha}_r}{h})\!\left(B_{1,\alpha,r} + \frac{\pi B_{2,\alpha,r}\log\log|t|}{\log|t|} + \frac{\pi B_{3,\alpha,r}}{\log|t|}\right) + \frac{\pi B_{4,\alpha,r}}{h\log|t|} + \frac{3.35\pi}{D_{r,\alpha}ht^{1/3}\log|t|} }
{1 - \frac{\log(2\pi)}{\log|t|} -  \frac{1-\log 2}{t^{2/3} \log|t|} - \frac{0.194\pi}{h} - \frac{9.908\pi}{h\log|t|}}. \]
and the right hand side is clearly decreasing in both $h$ and $t$ provided the denominator is positive.
The lower bound on $h_0$ precisely ensures this condition is met.
\end{proof}

\begin{corollary}[Littlewood]\label{cor:main-littlewood}
Suppose $0<\alpha<r$ and $\alpha<1/6$ and then for $t_0>10^{12}$ and
\[ h_0>\frac{97\pi}{500}\left(\frac{1 + \frac{4954}{97\log(t_0)}}{1 -\frac{1-\log 2}{t_0^{2/3}\log(t_0)}-\frac{\log(2\pi)}{\log(t_0)}}\right) \]
 we have that for all $t>t_0$ and $t^{2/3}>h>h_0$  
\[\frac{ N(t-h,t+h,\alpha)}{N(t+h)-N(t-h)}  < C_2^L(\alpha,r,h_0,t_0) \]
where $C_2^L(\alpha,r,h,t) $ is given by
\begin{align*}    \frac{2A_{1,\alpha,r}+\frac{2A_{4,\alpha,r}}{h}+ \pi(2A_{2,\alpha,r}+\frac{2A_{5,\alpha,r}}{h})\frac{\log\log|t|}{\log|t|}+\frac{2\pi A_{3,\alpha,r}}{\log|t|} + \frac{\pi A_{6,\alpha,r}}{h\log|t|} + \frac{\pi\left(0.335 + \frac{0.51}{\log|t|}\right)}{(r-\alpha)ht\log|t|}}
{1 - \frac{\log(2\pi)}{\log|t|} -  \frac{1-\log 2}{t^{2/3} \log|t|} - \frac{0.194\pi}{h} - \frac{9.908\pi}{h\log|t|}}
\end{align*}
is decreasing in both $h_0$ and $t_0$ for any fixed $\alpha$ and $r$.
A table of values for $C_2^L(\alpha,r,h_0,t_0)   $ is given in Table~\ref{tab:C2L}.
\end{corollary}
\begin{proof}
The proof is the same as the above.
\end{proof}
These allow us to conclude that if $B_{1,\alpha,r}<1/4$ or $A_{1,\alpha,r}<1/4$ then there exists a finite $h_0$, so that for $t$ sufficiently large a positive proportion of zeros in intervals of width at least $h_0$ will have real part in $[\alpha,1-\alpha]$.

One approach we use relies on bounds for $N(s,\alpha)$. As such we provide updated explicit versions of such bounds in Theorems~\ref{thm:circularregions1} and \ref{thm:circularregions2}.
We note that such bounds on $N(s,\alpha)$ are used in many applications directly. For example in log-free zero-density results (see for example \cite[Corollary~2.3]{Bellotti2025} and \cite[Lemmas 1 and 2]{Pintz2019}). Various approaches to zero-free regions use these types of results as well (see for example \cite{Montgomery1973,Bellotti2024,Ford2002}). Bounds like those in Theorems \ref{thm:main-jensen} and \ref{thm:main-littlewood} also have the potential to be used to improve other zero density results such as in \cite{Bellotti2025,Bellotti2023,Pintz2023} where there can be benefit in bounding the number of zeros in a rectangular region rather than a square one. 

We organize this paper as follows:
\begin{itemize}

    \item In Section~\ref{sec:jensenclassic} we apply Jensen's theorem to bound $N(s,\alpha)$. 
    In order to provide sharp bounds we provide methods to improve estimates for the integral inside the critical strip in Section~\ref{sec:inside}  and outside the critical strip in Section~\ref{sec:refinement-integration}. 

    \item In Sections~\ref{sec:rectangularjensen} and \ref{sec:littlewood} we provide bounds on $N(t-h,t+h,\alpha)$ arising from two different zero detectors, the first of which is better for $\alpha$ smaller, and the second better for $\alpha$ larger. 
    
    \item In order to implement the approach of Section~\ref{sec:littlewood}  in Section~\ref{sec:argintegrals} we provide an improved mechanism to bound the arg integrals which arising as secondary error terms in more standard zero density results.

       \item In Appendix~\ref{sec:background} we provide a number of explicit bounds for $\log|\zeta(s)|$ and $\log(\Gamma(s))$ which we needed elsewhere.
\end{itemize}

There are a number of avenues of future exploration which could improve the result here, particularly for small $t$.
Firstly, there are many choices for zero detectors we could use for $N(s,\alpha)$. In this article we currently use a classical version of Jensen's formula. It would be interesting to explore other detectors such as those in \cite{Ford2002,Ford2019,Bellotti2025,ThornerZaman2024,HeathBrown1992}. As these methods tend to involve estimating $|\frac{\zeta'}{\zeta}|$, or higher derivatives, the ability to bound $|\frac{\zeta'}{\zeta}|$ seems to be a limiting factor here.
However, following the approach of Section~\ref{sec:rectangularjensen} it is possible that even if they can't improve results for $N(1+it,\alpha)$, they could  nonetheless provide improvements for bounds on $N(t-h,t+h,\alpha)$.

A second natural direction of exploration is parts of our analysis away from the $1$-line to work at either $1+\eta$ or $1+\eta(t)$. This would almost certainly improve the qualitative values obtained for smaller $t$ by minimizing the coefficients of $\log\log|t|$ and the $O(1)$ terms at the expense of the coefficient of $\log|t|$.

In the other direction, for extremely large values of $t$, and small values of $\alpha$, producing a variant of Theorem \ref{thm:rectangularjensen} using Theorem \ref{thm:circularregions2} would asymptotically be better. 

If one wants an improved result for small values of $h$, then further developing Remark \ref{rem:smallh} is appropriate, as would be developing alternatives to, or simply not using, the optimizations of Section \ref{sec:intoutside2} since for short intervals bounding the integrand will provide a better bound than bounding its anti-derivative at the endpoints.

\subsection{AI Use Declaration}

All of the ideas, theorems, and proofs in this article were developed by the author. 

AI tools (Claude Opus 5 and Fable 5.1) were used to formalize the content of this article into Lean. In doing so it identified and corrected a number of minor errors. The errors were primarily grammatical, typographical, or missing implicit hypothesis. It also produced code to verify numerics tables. The current tables were produced using its code. The formalization and associated code will be made available on GitHub at \url{https://github.com/andrewfiori/ZeroDensityInShortIntervals}.

\section{Jensen bounds}\label{sec:jensenclassic}

In this section we provide initial bounds for the number of zeros in a circular region using Jensen's formula. Our focus in this section is setting up the argument. In Sections~\ref{sec:inside} and \ref{sec:refinement-integration} we will optimize aspects of the approach.

Denoting by $\rho$ a zero of the function $\zeta$ a standard way to express 
Jensen's formula is that 
\begin{equation}\label{eq:jensen}
\sum_{|\rho - 1+it| < r}  \log(r) - \log(|\rho - 1+it|) = \frac{1}{2\pi}\int_{-\pi/2}^{3\pi/2} \log|\zeta(1+it+re^{i\theta})|\,d\theta - \log|\zeta(1+it)| 
\end{equation}
This leads to
\begin{proposition}\label{prop:jensenbound}
For any $s \in \mathbb{C}$ and $0<\alpha<r$ we have
\[  N(s,\alpha) \leq \frac{1}{\log(r/\alpha)} \left( \frac{1}{2\pi}\int_{-\pi/2}^{3\pi/2} \log|\zeta(s+re^{i\theta})|\,d\theta - \log|\zeta(s)|  \right).\]
\end{proposition}
\begin{proof}
    Since each zero $\rho$ with $|s-\rho|\leq\alpha$ contributes at least $\log(r/\alpha)$ to the sum on the left, the result is immediate from Equation \eqref{eq:jensen}.
\end{proof}

The simplest bound for
\[  \frac{1}{2\pi}\int_{-\pi/2}^{\pi/2} \log|\zeta(1+it+re^{i\theta})|\,d\theta  \]
comes from simply taking the uniform bound of $\log|\zeta(1+it+re^{i\theta})| < \log\log|t|$ in this region which leads to a bound of $\log\log|t|/2 + O(1)$ for the integral.

If one uses instead  $\log|\zeta(1+it+re^{i\theta})| < \frac{r-\sigma}{r-1}(\log\log|t|) + \frac{\sigma-1}{r-1}\log(\zeta(r))$ from \cite[Section~4.1]{Fiori2026} one will instead have
\[ \left(\frac{1}{2}-\frac{1}{\pi}\right) \log\log|t| + \frac{\log(\zeta(1+r))}{\pi}  + O(1). \]
We shall provide better bounds for this term in Section~\ref{sec:refinement-integration} so we shall not dwell too long on this term for now.

The simplest bound for
\[  \frac{1}{2\pi}\int_{\pi/2}^{3\pi/2} \log|\zeta(1+it+re^{i\theta})|\,d\theta  \]
we consider comes from $\log|\zeta(1+it+re^{i\theta})|  < \frac{1}{3}r\cos(\theta)\log|t| + \log\log|t| + O(1)$
which leads to the bound
\[  \frac{1}{2\pi}\int_{\pi/2}^{3\pi/2} \log|\zeta(1+it+re^{i\theta})|\,d\theta < \frac{r}{3\pi}\log|t| + \frac{1}{2}\log\log|t| + O(1). \]
Bounding $-\log|\zeta(1+it)| < \log\log|t| + O(1)$ and combining terms thus gives:
\[  \sum_{|\rho - 1+it_0| < r}  \log(r) - \log(|\rho - 1+it_0|)  < \frac{r}{3\pi}\log|t| + (2-\frac{1}{\pi})\log\log|t| + \log(\zeta(r)) + O(1). \]
Counting the contribution from each zero we find that for $\alpha<r$ the number of zeros in a circle of radius $\alpha$ is bounded by
\begin{equation}\label{firstbound} \frac{\frac{r}{3\pi}\log|t| + (2-\frac{1}{\pi})\log\log|t| + \frac{\log(\zeta(1+r))}{\pi}  +\log(\zeta(r)) + O(1)}{\log(r/\alpha) } \end{equation}
It is clear that we can use a different bounds for $\zeta$, for example using the sharper bounds on $\zeta$ for $r<2/7$ one can improve this (for $r<2/7$) to 
\begin{equation}\label{shorterbound} \frac{\frac{r}{4\pi}\log|t| + (2-\frac{1}{\pi})\log\log|t|  +\frac{\log(\zeta(1+r))}{\pi}  +O(1)}{\log(r/\alpha)}. \end{equation}
The next section focuses on optimizing this bound further.

\subsection{Bounding Integrals Inside the Critical Strip}\label{sec:inside}

The bounds from Equation \eqref{firstbound} are already useful but we wish to refine it. Specifically, it is possible to improve the coefficient $\frac{r}{k\pi}$ by splitting the integral from $\pi/2$ to $3\pi/2$ based on the condition $r\cos(\theta)=1-\sigma_k$ where $\sigma_k$ are chosen lines where we have better bounds. This has been done previously for example in bounding $N(T)$, see for example \cite{BellottiWongFiori2025}.

The idea is to define piecewise functions
 $F_{c,r} : [-\pi,\pi] \rightarrow \mathbb{R}$ such that 
 \[ \log|\zeta(1+ it +re^{i\theta} )| \leq F_{1+ it,r}(\theta) \]
 and then integrate the function $F_{1+ it,r}(\theta) $.

To this end we shall define for $k\geq 0$
\begin{equation} \label{eq:sigmak} \sigma_k =  1-\frac{k+3}{2^{k+3}-2} \qquad  \sigma_{-k} = \frac{k+3}{2^{k+3}-2} \end{equation}
so that these are the vertical lines for which \cite{Yang2024}, combined with Lemma~\ref{lem:zetalessthanhalf} provides bounds.
We then define for $k\geq 0$
\[ v_k = \frac{1}{2^{k+3}-2}  \qquad  v_{-k} = \sigma_k-1/2 + \frac{1}{2^{k+3}-2} \]
which define the corresponding coefficient of $\log|t|$ for the bounds on $\zeta(\sigma_k+it)$.

Define also
\begin{align*}  v_k' &= 1  & v_k'' &= \log(1.546)& k>0\\
v_0'' &= \log(0.611) 
 &  v_k'' & = \log(1.546)+(\sigma_k-1/2)\log|\pi|& k<0
\end{align*}
respectively the coefficients of $\log\log|t|$ and the constant terms.

\begin{remark}\label{rem:altcoefficients}
It is also admissible to set either
\[ v_0'' =\log\left( 0.470795  + \frac{4.04972}{\log|T|} \right) \]
or
\[ v_0=\frac{27}{164},\qquad v_0' = 0,\quad  \text{  and  }\quad v_0'' = \log(66.7). \]
These values are used in all displayed tables.
\end{remark}

Now define
\begin{equation}\label{eq:mk} m_{k} = \frac{v_k-v_{k+1}}{\sigma_k - \sigma_{k+1}} \qquad   b_{k} = v_k - m_{k}\sigma_k \end{equation}
so that for $\sigma\in[\sigma_k,\sigma_{k+1}]$ the formula $m_k\sigma+b_k$ describes the coefficient of $\log|T|$ in the bound on $\log|\zeta(\sigma+it)|$.
Similarly, define
\begin{equation}\label{eq:mkp}m_k' =  \frac{v_k'-v_{k+1}'}{\sigma_k - \sigma_{k+1}} \qquad b_{k}' = v_k' - m_{k}'\sigma_k \end{equation}
and
\begin{equation}\label{eq:mkpp}m_k'' =  \frac{v_k''-v_{k+1}''}{\sigma_k - \sigma_{k+1}} \qquad b_{k}'' = v_k'' - m_{k}''\sigma_k \end{equation}

Now, for each $k$ with $r\geq 1-\sigma_k$ we define
\begin{equation}\label{eq:thetak} \theta_{k,r} = \arcsin((1-\sigma_k)/r) \end{equation}
and define $\theta_{-\infty,r} = \pi/2$.

Now set $K=\inf_{k}(r\geq 1-\sigma_k)$ where $K=-\infty$ is admissible. Finally, we define
\begin{align*} c_{1,r} &= (m_{K-1}+b_{K-1})(\pi/2-\theta_{K,r}) - m_{K-1}\,r\cos(\theta_{K,r})\\
&\qquad+ \sum_{k=K}^{\infty} \left[ (m_k+b_k)(\theta_{k,r}-\theta_{k+1,r}) + m_k\,r\left(\cos(\theta_{k,r})-\cos(\theta_{k+1,r})\right)\right] \\
c_{2,r} & = (m_{K-1}'+b_{K-1}')(\pi/2-\theta_{K,r}) - m_{K-1}'\,r\cos(\theta_{K,r})
\\&\qquad+ \sum_{k=K}^{\infty} \left[ (m_k'+b_k')(\theta_{k,r}-\theta_{k+1,r}) + m_k'\,r\left(\cos(\theta_{k,r})-\cos(\theta_{k+1,r})\right)\right] \\ 
c_{3,r} & = (m_{K-1}''+b_{K-1}'')(\pi/2-\theta_{K,r}) - m_{K-1}''\,r\cos(\theta_{K,r})
\\&\qquad+ \sum_{k=K}^{\infty} \left[ (m_k''+b_k'')(\theta_{k,r}-\theta_{k+1,r}) + m_k''\,r\left(\cos(\theta_{k,r})-\cos(\theta_{k+1,r})\right)\right]
\end{align*}
We note that every sum above is rapidly convergent, even if $K=-\infty$, and can be accurately computed.

The following proposition summarizes the result
\begin{proposition}\label{prop:jensen-easy}
    Suppose  $0<r<1$ and $t>13$. With $c_{1,r}$, $c_{2,r}$ and $c_{3,r}$ as defined above we have
    \begin{align*} \frac{1}{2\pi}\int_{ \pi/2}^{3\pi/2} \log|\zeta(1+it+re^{i\theta})|\,d\theta
    \leq& \frac{c_{1,r}}{\pi}\log|t|  + \frac{c_{2,r}}{\pi}\log\log|t| + \frac{c_{3,r}}{\pi}  \\& +\Ostar\!\left(\frac{5}{6(t-1)} + \frac{1}{(t-1)^2} + \frac{579}{8(t-1)^2\log(t-1)}\right)
    \end{align*}
    Some values for $c_{1,r}$, $c_{2,r}$ and $c_{3,r}$ are given in Table~\ref{tab:crvalues}.
\end{proposition}
\begin{proof}
   The key is that by construction between $\theta_{k+1,r}+\pi/2$ and $\theta_{k,r}+\pi/2$ that
   \begin{align*} F_{1+it,r}(\theta) &= \bigl((m_k+b_k)-m_kr\sin(\theta-\pi/2)\bigr)\log|t+r| \\&\qquad+ \bigl((m_k'+b_k')-m_k'r\sin(\theta-\pi/2)\bigr)\log\log|t+r|
   \\&\qquad+ \bigl((m_k''+b_k'')-m_k''r\sin(\theta-\pi/2)\bigr)
   \\&\qquad + \Ostar\!\left(\frac{2}{(t-r)^2} + \frac{579}{4(t-r)^2\log|t-r|}\right)
   \end{align*}
   bounds $\log|\zeta(1+ it +re^{i\theta})|$.

   The $\Ostar$ term is obtained by carefully considering the Taylor expansions of $\log|1+r/t|$ and $\log\left|1+\frac{\log|1+r/t|}{\log|t|}\right|$ in $\frac{r}{t}$ and telescoping the defining sums $c_{1,r}$ and $c_{2,r}$. 
\end{proof}
From what we have done already this would lead to the bound
    \begin{align*} N(1+it,\alpha) &< \frac{1}{\log(r/\alpha)}\Biggl( \frac{c_{1,r}}{\pi}\log|t| + \left(\frac{c_{2,r}}{\pi}-\frac{1}{\pi}\right)\log\log|t|
    \\& \qquad\qquad + \frac{1}{\pi}\log(\zeta(r)) + \frac{c_{3,r}}{\pi} +  \log(29.388) \Biggr).
    \end{align*}

\begin{remark}
For $r=1/2$ this reduces the value from $\frac{1}{6\pi}$ to $\frac{0.1480612309}{\pi}$, other values are in Table~\ref{tab:crvalues}.
\end{remark}

\begin{remark}
If one wants to bound the number of zeros in a circle of radius $\alpha$ using \eqref{firstbound} the best choice of $r$, asymptotically, would be $r=e\alpha$ as this minimizes the ratio $\frac{s\alpha}{\log(s)}$ which is the coefficient of $\log|t|$.

The optimal choice for $r$ when using $c_r$ is more complex as the derivative of $c_r$ with respect to $r$ is complex, see Table~\ref{tab:crvalues}.
\end{remark}

\begin{remark}
Extending the result past $r>1$ or for $s=c+it$ with $c>1$ can be done and is used in some applications. Typically this is done to decrease the coefficient of $\log\log|t|$ by moving the $-\log(\zeta(c+it)$ term far enough to the right so that it becomes $O(\log(\zeta(c)))$. This will come at the expense of increasing the coefficient of $\log|t|$. 
\end{remark}

Using instead the Richert type bound we get instead
\begin{proposition}\label{prop:Richert}
Assume that $r>0$ and that for $t\in [T-r,T+r]$ and $(1-\sigma)<r$ we have a bound
\[ \log|\zeta(\sigma+it) | < B(1-\sigma)^{3/2}\log|T| + \frac{2}{3} \log\log|T| + A \]
then
\begin{equation}\label{secondbound}
\frac{1}{2\pi}\int_{\pi/2}^{3\pi/2} \log|\zeta(1+it+re^{i\theta})|\,d\theta \leq 
\frac{E_{1/2} Br^{3/2}}{\pi}\log|T| + \frac{1}{3}\log\log|T| + \frac{A}{2}
\end{equation}
where 
\[ E_{1/2} = \int_0^{\pi/2} \cos(\theta)^{3/2}  =  \frac{1}{3} \int_0^{\pi/2} \frac{1}{\sqrt{1-2\sin^2(\theta/2)}} \sim 0.874019\ldots \]
comes from an elliptic integral of the first kind.
\end{proposition}
\begin{remark}
For \eqref{secondbound} with a fixed $\alpha$ the best choice for $r$ is $r=e^{2/3}\alpha$. This would give the bound
\[ N(1+it,\alpha) < 
\frac{3}{2}\left(\frac{ E_{1/2}e B \alpha^{3/2} }{\pi} \log|t| + \left(1-\frac{1}{2\pi} + \frac{1}{3}\right)\log\log|t| + O_r(1) \right)
\]
For example with $r=\frac{1}{2e^{1/3}}$ and $\alpha=\frac{1}{2e}$ we have a lead term $\frac{0.623846}{2\pi}\log|t|$
which is weaker than the previous result.  
Indeed, to  see an improvement you would need $\alpha$ to be quite small.

We do however note that the constant $\frac{3 e}{2 \pi}E_{1/2} \sim 1.134375 < 1.3478$, which was the constant obtained by Ford in \cite[Lemma~4.2]{Ford2002,Ford2019}.
Our advantage over Ford comes entirely from our ability to work at $1+it$, whereas he must work at $1+0.6421\alpha + it$ in order to handle a $\frac{\zeta'}{\zeta}$ term.

Note, although working on the line $1+it$ is good for our main term, it would be worthwhile to explore the application of these methods along $1+\eta + it$ which may improve results for smaller values of $t$.
\end{remark}

\begin{remark}
    We make some final remarks about minor improvements one could make to the main term integral
    \[\int_{\pi/2}^{3\pi/2} \log|\zeta(1+it+re^{i\theta})|\,d\theta . \]
\begin{enumerate}
    \item In principle it would be best to pick the critical-line bound based on the height $t$ at which one is working.
    The critical-line bound could give small savings on the main term and the secondary $\log\log$ term, at the expense of constant terms.
    
    \item We have been wasteful in considering $k=3\cdots\infty$. For small $t$ at some point (for instance when you enter the zero-free region) these bounds are actually worse than other available bounds. For large $t$ they will (eventually) be worse than the Richter bound.

    For smaller $t$, if one is using $\alpha$ that are constant size, then the savings here are modest since the width of the region where there would be an improvement is quite thin. 
    
    \item For $r$ very small, say $O(\log\log|t|/\log|t|)$ or  $O(\log\log|t|^{2/3}/\log|t|^{2/3})$, it should be noted that the $\log\log|t|$ terms become the dominant error term. On this scale other methods are likely to be superior to these.
\end{enumerate}
    
\end{remark}

\subsection{Integration outside the critical strip}\label{sec:refinement-integration}
We now focus our attention on optimizing the secondary error terms.

For the bounds in Equation \eqref{firstbound} we bounded
\[ \int_{-\pi/2}^{\pi/2}  \log|\zeta(1+it+re^{i\theta})|\,d\theta  \]
by using the bounds on $|\zeta(s)|$ outside the critical strip.

The bound we obtain outside the critical strip is far from optimal. In particular, for $r$ of constant size the integral can be bounded independent of $t$.

\begin{proposition}\label{prop:integraloutside}
For $t>2$ and $0<r<t$, we have that
\[  \frac{1}{2\pi}\int_{-\pi/2}^{\pi/2}  \log|\zeta(1+it+re^{i\theta})|\,d\theta
\leq c_{4,r} \]
where 
\[ c_{4,r} = \frac{1}{2}\log(\zeta(1+r)) -\frac{r}{\pi}\int_{0}^{1} \frac{\zeta'(1+ru)}{\zeta(1+ru)}\arcsin(u)\,du .\]
Moreover, if $r\leq 2$, we have
\[ c_{4,r} <  \frac{1}{2}\log(\zeta(1+r)) + \frac{\log(2)}{2} - \frac{(\pi-2)r}{2\pi}\gamma  + \frac{r^2(\gamma^2+2\gamma_1)}{8} \]
where $\gamma$ is the Euler-Mascheroni constant and $\gamma_1$ is a Stieltjes constant.

Table~\ref{tab:crvalues} provides examples of computed values for $c_{4,r}$. 
\end{proposition}
\begin{proof}
We have for $\theta\in (-\pi/2,\pi/2)$ that 
\[ \log|\zeta(1+it+re^{i\theta})|\,d\theta  <  \log(\zeta(1+r\cos(\theta))) \]
and hence
\[ \frac{1}{2\pi}\int_{-\pi/2}^{\pi/2}  \log|\zeta(1+it+re^{i\theta})|\,d\theta < 
\frac{1}{\pi}\int_{0}^{\pi/2} \log(\zeta(1+r\cos(\theta)))\,d\theta
\]
 provided this integral converges, which one can verify that it does (say by an integral comparison test).
Through the substitution $u=\cos(\theta)$ we have
\[ \frac{1}{\pi}\int_{0}^{\pi/2} \log(\zeta(1+r\cos(\theta)))\,d\theta
= \frac{1}{\pi}\int_{0}^{1} \frac{\log(\zeta(1+ru))}{\sqrt{1-u^2}}du.
\]
And integration by parts yields
\begin{align*}
\frac{1}{\pi}\int_{0}^{1} \frac{\log(\zeta(1+ru))}{\sqrt{1-u^2}}du &= 
\frac{1}{\pi}\lim_{u\rightarrow 1-} \log(\zeta(1+ru))\arcsin(u)
\\&\qquad -
\frac{1}{\pi}\lim_{u\rightarrow 0+} \log(\zeta(1+ru))\arcsin(u)\\&\qquad -
\frac{1}{\pi}\int_{0}^{1} r\frac{\zeta'(1+ru)}{\zeta(1+ru)}\arcsin(u)\,du \\
&= \frac{1}{2}\log(\zeta(1+r)) -\frac{r}{\pi}\int_{0}^{1} \frac{\zeta'}{\zeta}(1+ru)\arcsin(u)\,du.
\end{align*}
Although $\frac{\zeta'}{\zeta}$ has a simple pole at $1$, $\arcsin$ has a simple zero at $0$, so that
$\frac{\zeta'}{\zeta}(1+ru))\arcsin(u)$ extends to an analytic function at $0$, and hence $\frac{\zeta'}{\zeta}(1+ru))\arcsin(u)$ is continuous on $[0,1]$. From which it follows that the integral exists, and we have the first claim.

By Lemma~\ref{lemma:logderzetabound}  we have the bound
\[ -\frac{\zeta'(1+z)}{\zeta(1+z)} < \frac{1}{z} - \gamma +  (\gamma^2+2\gamma_1)z \]
where $\gamma$ is the Euler-Mascheroni constant and $\gamma_1$ is a Stieltjes constant, this gives
\begin{align*} -\frac{r}{\pi}\int_{0}^{1} &\frac{\zeta'}{\zeta}(1+ru))\arcsin(u)\,du \\\qquad\qquad&< \frac{1}{\pi}\int_{0}^{1} \frac{\arcsin(u)}{u} - r\gamma\arcsin(u) + (\gamma^2+2\gamma_1)r^2u\arcsin(u)\,du \\
&= \frac{\log(2)}{2} - \frac{(\pi-2)r}{2\pi}\gamma  + \frac{r^2(\gamma^2+2\gamma_1)}{8}\qedhere
\end{align*}
\end{proof}

We now state our final Theorems on bounds for circular regions, the first being based on sub-convexity bounds for $\zeta$ and the second based on Richert bounds.
\begin{theorem}\label{thm:circularregions1}
Let $1>r>\alpha$ and $t>13$ then
\begin{align*} N(1+it,\alpha) \leq&
\frac{C_{1,\alpha,r}}{\pi} \log|t| +
{C_{2,\alpha,r}} \log\log|t| +
C_{3,\alpha,r} \\&\quad+  \Ostar\!\left(\frac{1}{\log(r/\alpha)}\left(\frac{5}{6(t-1)} + \frac{1}{(t-1)^2} + \frac{579}{8(t-1)^2\log(t-1)}\right)\right)
\end{align*}
where
\begin{align*}
 C_{1,\alpha,r} &= \frac{c_{1,r}}{\log(r/\alpha)}\\
 C_{2,\alpha,r} &= \frac{c_{2,r}}{\log(r/\alpha)\pi} + \frac{1}{\log(r/\alpha)}\\
 C_{3,\alpha,r} &= \frac{c_{3,r}}{\pi\log(r/\alpha)} + \frac{c_{4,r}}{\log(r/\alpha)} + \frac{\log(29.388)}{\log(r/\alpha)}.
 \end{align*}
Table~\ref{tab:CRvalues} contains computed values of the constants where we optimized $r$ so as to minimize $C_{1,\alpha,r}$. Table~\ref{tab:crvalues2} contains the associated values for $c_{i,r}$.
\end{theorem}
\begin{proof}
Jensen's formula gives us that
\[ N(1+it,\alpha)  <   \frac{1}{\log(r/\alpha)}\left(  \frac{1}{2\pi}\int_{-\pi/2}^{3\pi/2} \log|\zeta(1+it+re^{i\theta})|\,d\theta - \log|\zeta(1+it)|   \right) \]
The result now follows from a direct application of Propositions~\ref{prop:jensen-easy}, \ref{prop:integraloutside} and the bound
\[ -\log|\zeta(1+it)| < \log\log|t| + \log(29.388) \]
from \cite{Leong2024}.  
\end{proof}

To apply the above for more general values of $\alpha$ than are in the tables we give also
\begin{proposition}\label{prop:Cscaling}
Let $1>r>\alpha_0>\alpha$ and $r\geq e\alpha_0$ then
\[  C_{1,\alpha,(\alpha/\alpha_0)r} \leq C_{1,\alpha_0,r}\frac{\alpha}{\alpha_0}\qquad C_{2,\alpha,(\alpha/\alpha_0)r} \leq \frac{3}{2\log(r/\alpha_0)} .\]
\end{proposition}
\begin{proof}
The claim about $C_{1,\alpha,(\alpha/\alpha_0)r}$ is a consequence of the claim
\[  c_{1,(\alpha/\alpha_0)r}  \leq (\alpha/\alpha_0)c_{1,r}. \]
This claim follows from the observation that $\frac{d}{dr} \frac{c_{1,r}}{r}$ is positive.
To see this note that $\theta_{k,r} - \theta_{k+1,r}$ decreases as $r$ increases, but $m_k+b_k\leq 0$.
Additionally, although $\sum_{k=K}^{\infty} (-\cos(\theta_{k,r}) + \cos(\theta_{k+1,r}))$ is constant as $r$ changes, but, as $r$ increases, the weight shifts to lower values of $k$ for which $m_k$ is larger.
With these observations, the first claim now follows by inspection.

For the second claim, note again that $((\alpha/\alpha_0)r)/\alpha=r/\alpha_0$, so that
$C_{2,\alpha,(\alpha/\alpha_0)r}=\bigl(c_{2,(\alpha/\alpha_0)r}/\pi+1\bigr)/\log(r/\alpha_0)$. Telescoping the $c_{2,r}$ sum then gives the result.
\end{proof}

\begin{theorem}\label{thm:circularregions2}
Let $0<\alpha\leq 1$ and $t>2$. Assume that for $\tau\in [t-e^{2/3}\alpha,t+e^{2/3}\alpha]$ and
$(1-\sigma)< e^{2/3}\alpha$ we have a bound
\[ \log|\zeta(\sigma+i\tau) | < B(1-\sigma)^{3/2}\log|\tau| + \frac{2}{3} \log\log|\tau| + A \]
then
\[ N(1+it,\alpha) \leq
\frac{\tilde{C}_{1,\alpha}}{\pi} \log|t| +
2 \log\log|t| +
\tilde{C}_{3,\alpha}
\]
where
\begin{align*}
 \tilde{C}_{1,\alpha} & = \frac{3E_{1/2}e B\alpha^{3/2}}{2},  \\ 
 \tilde{C}_{3,\alpha} & = \frac{3A}{4} + \frac{3c_{4,e^{2/3}\alpha}}{2} + \frac{3\log(29.388)}{2}\\
 & <  \frac{3A}{4}  + \frac{3}{4}\log(\zeta(1+e^{2/3}\alpha)) + \frac{3\log(2)}{4}+ \frac{3\log(29.388)}{2}.
\end{align*}
\end{theorem}
\begin{proof}
Follows as above using Proposition~\ref{prop:Richert} in place of Proposition~\ref{prop:jensen-easy},
with $r=e^{2/3}\alpha$, so that $\log(r/\alpha)=\frac{2}{3}$.
\end{proof}

\begin{table}[h]
\caption{Sample values for $c_{1,r}$, $c_{2,r}$, $c_{3,r}$ as defined in Proposition~\ref{prop:jensen-easy}, $c_{4,r}$ as defined in Proposition~\ref{prop:integraloutside}, and $c_{5,r}$ as defined in Proposition~\ref{prop:outside3}.}\label{tab:crvalues}
\begin{tabular}{|c|ccccc|}
\hline
r  & $c_{1,r}$& $c_{2,r}$& $c_{3,r}$& $c_{4,r}$& $c_{5,r}$\\
\hline
$ 1 $&$0.4260491$&$1.320694$&$1.233886$&$0.5099522$&$0.8331601$\\
$ 1/2 $&$0.1467868$&$0.9393432$&$3.061478$&$0.7795712$&$1.126439$\\
$ 1/3 $&$0.08122972$&$1.491029$&$0.9846382$&$0.9546469$&$1.271554$\\
$ 1/4 $&$0.05402652$&$1.570797$&$0.6843504$&$1.084244$&$1.360125$\\
$ 1/5 $&$0.03956564$&$1.570797$&$0.6843504$&$1.187130$&$1.420588$\\
$ 1/6 $&$0.03071227$&$1.570797$&$0.6843504$&$1.272441$&$1.464849$\\
$ 1/7 $&$0.02516876$&$1.570797$&$0.6843504$&$1.345309$&$1.498839$\\
$ 1/8 $&$0.02110448$&$1.570797$&$0.6843504$&$1.408902$&$1.525868$\\
$ 1/9 $&$0.01803813$&$1.570797$&$0.6843504$&$1.465317$&$1.547944$\\
$ 1/10 $&$0.01569390$&$1.570797$&$0.6843504$&$1.516009$&$1.566357$\\
$ 1/11 $&$0.01393216$&$1.570797$&$0.6843504$&$1.562034$&$1.581977$\\
$ 1/12 $&$0.01249892$&$1.570797$&$0.6843504$&$1.604178$&$1.595417$\\
$ 1/13 $&$0.01129907$&$1.570797$&$0.6843504$&$1.643045$&$1.607118$\\
$ 1/14 $&$0.01028337$&$1.570797$&$0.6843504$&$1.679108$&$1.617407$\\
$ 1/15 $&$0.009416022$&$1.570797$&$0.6843504$&$1.712745$&$1.626535$\\
$ 1/16 $&$0.008670644$&$1.570797$&$0.6843504$&$1.744261$&$1.634694$\\
$ 1/32 $&$0.003703026$&$1.570797$&$0.6843504$&$2.085161$&$1.702482$\\
$ 1/64 $&$0.001626574$&$1.570797$&$0.6843504$&$2.428881$&$1.743167$\\
$ 1/128 $&$7.256286\cdot 10^{-4}$&$1.570797$&$0.6843504$&$2.774023$&$1.766930$\\
$ 1/256 $&$3.277390\cdot 10^{-4}$&$1.570797$&$0.6843504$&$3.119880$&$1.780529$\\
$ 1/512 $&$1.495313\cdot 10^{-4}$&$1.570797$&$0.6843504$&$3.466095$&$1.788188$\\
$ 1/1024 $&$6.880172\cdot 10^{-5}$&$1.570797$&$0.6843504$&$3.812489$&$1.792449$\\
$ 1/2048 $&$3.188568\cdot 10^{-5}$&$1.570797$&$0.6843504$&$4.158973$&$1.794794$\\
\hline
\end{tabular}
\end{table}
\begin{table}
\caption{Values for $C_{1,\alpha,r}$, $C_{1,\alpha,r}$, and $C_{1,\alpha,r}$ as defined in Theorem~\ref{thm:circularregions1}. Values, including $r$, are rounded up.}\label{tab:CRvalues}
\begin{tabular}{|cc|ccc|c|}
\hline
$\alpha$ & $r$ & $C_{1,\alpha,r}$ & $C_{2,\alpha,r}$ & $C_{3,\alpha,r}$   & $r/\alpha$\\
\hline
$ 1/3 $&$0.6490511$&$0.3264740$&$1.965836$&$7.446222$&$1.947154$\\
$ 1/4 $&$0.5022756$&$0.2117515$&$1.858502$&$7.369401$&$2.009103$\\
$ 1/5 $&$0.3940064$&$0.1538013$&$2.093277$&$7.054466$&$1.970032$\\
$ 1/6 $&$0.3242696$&$0.1170874$&$2.225533$&$6.965111$&$1.945618$\\
$ 1/7 $&$0.2897366$&$0.09333569$&$2.120321$&$6.530988$&$2.028157$\\
$ 1/8 $&$0.2606756$&$0.07783097$&$2.040920$&$6.345352$&$2.085405$\\
$ 1/9 $&$0.2276378$&$0.06616898$&$2.091395$&$6.588795$&$2.048740$\\
$ 1/10 $&$0.2030836$&$0.05706802$&$2.117307$&$6.744948$&$2.030836$\\
$ 1/11 $&$0.1851792$&$0.04988792$&$2.108330$&$6.776810$&$2.036971$\\
$ 1/12 $&$0.1729682$&$0.04420265$&$2.054067$&$6.646159$&$2.075618$\\
$ 1/13 $&$0.1668192$&$0.03972128$&$1.937723$&$6.291698$&$2.168650$\\
$ 1/14 $&$0.1550286$&$0.03611895$&$1.935706$&$6.329759$&$2.170400$\\
$ 1/15 $&$0.1433184$&$0.03302353$&$1.959855$&$6.457329$&$2.149775$\\
$ 1/16 $&$0.1333413$&$0.03033976$&$1.979559$&$6.567526$&$2.133460$\\
$ 1/32 $&$0.06834560$&$0.01242107$&$1.916792$&$6.771429$&$2.187059$\\
$ 1/64 $&$0.03512879$&$0.005269206$&$1.851512$&$6.944143$&$2.248243$\\
$ 1/128 $&$0.01817011$&$0.002295522$&$1.777141$&$7.052067$&$2.325774$\\
$ 1/256 $&$0.009413507$&$0.001021403$&$1.705384$&$7.139335$&$2.409858$\\
$ 1/512 $&$0.004727625$&$4.607252\cdot 10^{-4}$&$1.696848$&$7.492183$&$2.420544$\\
$ 1/1024 $&$0.002374581$&$2.099090\cdot 10^{-4}$&$1.688169$&$7.840865$&$2.431570$\\
$ 1/2048 $&$0.001192880$&$9.643198\cdot 10^{-5}$&$1.679292$&$8.184763$&$2.443018$\\
\hline
\end{tabular}
\end{table}
\begin{table}
\caption{Values for $c_{1,r}$, $c_{2,r}$, $c_{3,r}$ as defined in Proposition~\ref{prop:jensen-easy}, and $c_{4,r}$ as defined in Proposition~\ref{prop:integraloutside}, for the unrounded $r$ values from Table~\ref{tab:CRvalues}. All values are rounded up.}\label{tab:crvalues2}
\begin{tabular}{|c|cccc|}
\hline
r  & $c_{1,r}$& $c_{2,r}$& $c_{3,r}$& $c_{4,r}$\\
\hline
$0.6490511$&$0.2175520$&$0.9738017$&$2.853749$&$0.6729643$\\
$0.5022756$&$0.1477365$&$0.9319669$&$3.089081$&$0.7776710$\\
$0.3940064$&$0.1042849$&$1.317413$&$1.638224$&$0.8812288$\\
$0.3242696$&$0.07793089$&$1.511949$&$0.9058834$&$0.9668954$\\
$0.2897366$&$0.06600019$&$1.568710$&$0.6922055$&$1.017316$\\
$0.2606756$&$0.05720288$&$1.570797$&$0.6843504$&$1.065177$\\
$0.2276378$&$0.04745803$&$1.570797$&$0.6843504$&$1.127225$\\
$0.2030836$&$0.04042969$&$1.570797$&$0.6843504$&$1.180019$\\
$0.1851792$&$0.03549344$&$1.570797$&$0.6843504$&$1.223032$\\
$0.1729682$&$0.03227938$&$1.570797$&$0.6843504$&$1.254995$\\
$0.1668192$&$0.03074843$&$1.570797$&$0.6843504$&$1.272011$\\
$0.1550286$&$0.02798899$&$1.570797$&$0.6843504$&$1.306581$\\
$0.1433183$&$0.02527499$&$1.570797$&$0.6843504$&$1.343779$\\
$0.1333413$&$0.02298980$&$1.570797$&$0.6843504$&$1.378087$\\
$0.06834559$&$0.009720201$&$1.570797$&$0.6843504$&$1.700612$\\
$0.03512879$&$0.004268841$&$1.570797$&$0.6843504$&$2.027367$\\
$0.01817011$&$0.001937541$&$1.570797$&$0.6843504$&$2.353894$\\
$0.009413510$&$8.983930\cdot 10^{-4}$&$1.570797$&$0.6843504$&$2.681106$\\
$0.004727620$&$4.072770\cdot 10^{-4}$&$1.570797$&$0.6843504$&$3.024609$\\
$0.002374580$&$1.865119\cdot 10^{-4}$&$1.570797$&$0.6843504$&$3.368478$\\
$0.001192880$&$8.613635\cdot 10^{-5}$&$1.570797$&$0.6843504$&$3.712486$\\
\hline
\end{tabular}
\end{table}

\section{Rectangular intervals}\label{sec:rectangularjensen}

We now work to bound the number of zeros in rectangular regions near the one line. We recall that $N(t-h,t+h,\alpha)$ denotes the number of zeros in the rectangle with vertices $1+it-ih$,  $1+it+ih$, $1+it-ih-\alpha$, and  $1+it+ih-\alpha$.
We shall ultimately provide two methods to obtain similar results, each method performs better for a different range of values $\alpha$.

To bound $N(t-h,t+h,\alpha)$ we define
\begin{equation}\label{eq:hatalpha} \hat{\alpha}_r = \sqrt{r^2-\alpha^2} \end{equation}
and consider the integral
\begin{equation}
    \int_{-h-\hat{\alpha}_r}^{h+\hat{\alpha}_r} \left(\frac{1}{2\pi}\int_0^{2\pi} \log|\zeta(1+i(t+x)+re^{i\theta})|\,d\theta - \log|\zeta(1+i(t+x))|\right)\,dx
\end{equation}
so that each zero lying within a radius $r$ of a point on the line $i(t-h-\hat{\alpha}_r-r)$ to $(t-h-\hat{\alpha}_r-r)$ contributes to the integral. We get the following:
 For $0<\alpha<r$ write
\[ D_{r,\alpha} := 2r\sqrt{1-(\alpha/r)^2}-2\alpha\tan^{-1}(\sqrt{(r/\alpha)^2-1}) \]
for the quantity appearing repeatedly as a denominator below.
\begin{proposition}\label{prop:jensenrectangle}
Suppose $0<\alpha<r$ and $0<h<t-\hat{\alpha}_r-r$.
\begin{align*} N(t-h,t+h,\alpha) \leq \frac{1}{D_{r,\alpha}}\int_{-h-\hat{\alpha}_r}^{h+\hat{\alpha}_r} \Biggl(&\frac{1}{2\pi}\int_0^{2\pi} \log|\zeta(1+i(t+x)+re^{i\theta})|\,d\theta
\\& - \log|\zeta(1+i(t+x))|\Biggr)\,dx.
\end{align*}
\end{proposition}
\begin{proof}
From Equation~\ref{eq:jensen} we have
\begin{align*} \sum_{|\rho - 1+it| < r}  \log(r/|\rho - 1+it|) &= \frac{1}{2\pi}\int_0^{2\pi} \log|\zeta(1+i(t+x)+re^{i\theta})|\,d\theta
\\&\quad - \log|\zeta(1+i(t+x))|. \end{align*}
As such we have
\begin{align*} &  \int_{-h-\hat{\alpha}_r}^{h+\hat{\alpha}_r} \left(\frac{1}{2\pi}\int_0^{2\pi} \log|\zeta(1+i(t+x)+re^{i\theta})|\,d\theta - \log|\zeta(1+i(t+x))|\right)\,dx \\
&\quad= \int_{-h-\hat{\alpha}_r}^{h+\hat{\alpha}_r}  \sum_{|\rho - 1+it| < r}  \log(r/|\rho - 1+it|)\,dt.
\end{align*}
Exchanging the integral and the sum gives that this is larger than
\[  \sum_{\substack{1-\Re(\rho)<\alpha\\-h<\Im(\rho)-t_0<h }}
\int_{\Im(\rho)-\hat{\alpha}_r}^{\Im(\rho)+\hat{\alpha}_r} \log(r/|\rho - 1+it|)dt  \}
\]
So that the integral detects each zero with weight 
\begin{align*} 
\int_{\Im(\rho)-\hat{\alpha}_r}^{\Im(\rho)+\hat{\alpha}_r} \log(r) -\log(|\rho - 1+it|)dt 
&= \int_{-\hat{\alpha}_r}^{\hat{\alpha}_r} \log(r) - \log((\Re(\rho) - 1)^2+t^2)dt \\
&> \int_{-\hat{\alpha}_r}^{\hat{\alpha}_r} \log(r) - \frac{1}{2}\log(\alpha^2+t^2)dt.
\end{align*}
Since
\[\int_{-\hat{\alpha}_r}^{\hat{\alpha}_r} \log(r) - \frac{1}{2}\log(\alpha^2+t^2)dt
=    D_{r,\alpha}
\]
 we have the claimed result.   
\end{proof}

It follows simply by integrating equation \eqref{firstbound} that for $r>\alpha$
\[ N(t-h,t+h,\alpha) < (2h+2\hat{\alpha}_r)\frac{\frac{1}{3\pi}\log|t| + O(\log\log|t|)}{2\sqrt{1-(\alpha/r)^2}-2(\alpha/r)\tan^{-1}(\sqrt{(r/\alpha)^2-1})} \]

\begin{remark}
    In contrast to the situation with circular region for a fixed $\alpha$ the naive optimal value for $r$ would appear to be to simply take the largest possible $r$, however, in practice better values of $c_r$ for smaller $r$ eventually provide a benefit. See Table~\ref{tab:rectangularjensen}.
 \end{remark}

Taking $\alpha=1/8$ and  $r=1/2$ gives
\[ N(t-h,t+h,1/8) < (2h+\sqrt{15}/4)\frac{\frac{c_{1,1/2}}{\pi}\log|t| + 1.5\log\log|t| + O(1)} {\sqrt{15}/4-(1/4)\tan^{-1}(\sqrt{15})}\]
so that we get the dominant term
\[ \frac{0.45964}{\pi} h \log|t|. \]
This will tell us that for $t$ and $h$ large enough that at least $8.07\%$ of zeros have real part in the interval $[1/8,7/8]$.

\begin{remark}\label{rem:smallh}
It is possible to tune differently the range of integration, for example 
integrating from $-h$ to $+h$ completely removes the $2\hat{\alpha}_r$ term but gives instead
\[ N(t-h,t+h,\alpha) < (4h)\frac{\frac{1}{3\pi}\log|t| + (1-\frac{1}{\pi} + 1)\log\log|t| + C}{2\sqrt{1-(\alpha/r)^2}-2(\alpha/r)\tan^{-1}(\sqrt{(r/\alpha)^2-1})}  \]
which is clearly worse for large $h$, but in the example $r=1/2$, $\alpha=1/8$ is an improvement for $h<\sqrt{15}/{8}$
\end{remark}

\subsection{Integration outside the critical strip: Part II}\label{sec:intoutside2}

As with circular regions there is advantage to taking a more refined approach to integrals outside the critical strip. That is, rather than integrating our bound on the error term arising in Proposition~\ref{prop:integraloutside} we shall exchange the order of integration in this region. That is we shall integrate
\[  \log|\zeta(1+it)| + 
\frac{1}{2\pi}\int_{-\pi/2}^{\pi/2} \log|\zeta(1+it+re^{i\theta})|\,d\theta. \]

\begin{proposition}\label{prop:outside2}
 Suppose $h\geq0$ and $T>h+\hat{\alpha}_r$ then
\[   \frac{1}{2\pi}\int_{T-h-\hat{\alpha}_r}^{T+h+\hat{\alpha}_r}  \log|\zeta(1+it)|dt \leq \frac{\Phi(1)}{\pi}  \]
where
\begin{equation}\label{eq:phis} \Phi(s) = \sum_{n>1} \frac{\Lambda(n)}{\log(n)^2 n^s}
\end{equation}
is an anti-derivative of $-\log(\zeta(s))$.
Moreover,
\[ \Phi(1) < 1.7975699586287395. \]
\end{proposition}
\begin{proof}
By the Fundamental Theorem of Calculus for all $\epsilon > 0$ we have
\begin{align*} \int_{T-h-\hat{\alpha}_r}^{T+h+\hat{\alpha}_r}  \log|\zeta(1+\epsilon+it)|dt 
&= \Re(\Phi(1+\epsilon+i(t+h+\hat{\alpha}_r)) - \Phi(1+\epsilon+i(t-h-\hat{\alpha}_r)))
\\& \leq 2\Phi(1+\epsilon).
\end{align*}
Since $\log(\zeta(1+\epsilon+it))$ is uniformly continuous for $\epsilon\in[0,1]$ and $t\in [T-h-\hat{\alpha}_r,T+h+\hat{\alpha}_r]$ and since $\Phi(s)$ is continuous on $[1,\infty)$, the result follows by taking the limit.

The bound for $\Phi(1)$ comes from an explicit rigorous computation of the defining integral.
\end{proof}

\begin{proposition}\label{prop:outside3}
Suppose $T>h+\hat{\alpha}_r$ then
   \begin{align*} \left|\int_{-h-\hat{\alpha}_r}^{h+\hat{\alpha}_r}\right.&  \left.\frac{1}{2\pi}\int_{-\pi/2}^{\pi/2} \log|\zeta(1+it+ix+re^{i\theta})|\,d\theta \,dx\right| \leq c_{5,r}
   \end{align*}
where $c_{5,r} =  \Phi(1+r) +  \frac{2r}{\pi}\int_{0}^{1} \log(\zeta(1+ru))\arcsin(u)\,du.$
Moreover, if $0 < r \leq 2$ then
\[  c_{5,r} \leq
\Phi(1+r)  + \left(\frac{-4}{\pi} +1+ \frac{\gamma r}{4} - \log(r)+ \frac{2\log(2 r)}{\pi} \right)r
\]
Values for $c_{5,r}$ are given in Table~\ref{tab:crvalues}
\end{proposition}
\begin{proof}
By exchanging the order of integration we have
\begin{align*} -\int_{-h-\hat{\alpha}_r}^{h+\hat{\alpha}_r}& \frac{1}{2\pi}\int_{-\pi/2}^{\pi/2} \log|\zeta(1+it+ix+re^{i\theta})|\,d\theta\,dx\\
& = -\frac{1}{2\pi}\Re\left(\int_{-\pi/2}^{\pi/2}\int_{-h-\hat{\alpha}_r}^{h+\hat{\alpha}_r} \log(\zeta(1+it+ix+re^{i\theta}))\,dx\,d\theta\right) \\
&=  \frac{1}{2\pi} \Re\left(\int_{-\pi/2}^{\pi/2} \Phi(1+i(t+h+\hat{\alpha}_r)+re^{i\theta}) - \Phi(1+i(t-h-\hat{\alpha}_r)+re^{i\theta})\,d\theta\right)\\
&\leq \frac{1}{\pi} \int_{-\pi/2}^{\pi/2}  \Phi(1+r\cos(\theta))\,d\theta\\
&= \frac{2}{\pi} \int_{0}^{1} \frac{\Phi(1+ru)}{\sqrt{1-u^2}}\,du\\
&= \Phi(1+r) +  \frac{2r}{\pi} \int_{0}^{1} \log(\zeta(1+ru))\arcsin(u)\,du
\end{align*} 
Noting that $\lim_{u\rightarrow0} \log(\zeta(1+ru))\arcsin(u) = 0$ we conclude the integral is convergent and get the first claim.

For the second claim use the bound from Lemma~\ref{lemma:logderzetabound}
\[  \log(\zeta(1+ru)) \leq  -\log(ru) + \gamma ru  \]
to obtain
\begin{align*}
 \int_0^1\log(\zeta(1+ru))\arcsin(u)\,du 
 &\leq \int_{0}^{1} ( -\log(ru) + \gamma ru )\arcsin(u)\,du \\
&= -2 + \frac{\pi}{2} + \frac{\gamma\pi r}{8} - \frac{\pi \log(r)}{2} +\log(2 r).
\end{align*}
\end{proof}

\begin{theorem}\label{thm:rectangularjensen}
Suppose $0<\alpha<r<1$, $0<h$ and $t-h-\hat{\alpha}_r>13$,  then
\begin{align*} N(t-h,t+h,\alpha) &\leq  (2h+2\hat{\alpha}_r)\left(\frac{B_{1,\alpha,r}}{\pi} \log|t|
+ B_{2,\alpha,r}\log\log|t| +  B_{3,\alpha,r} \right) + B_{4,\alpha,r} \\&\quad+ \Ostar\!\biggl(\frac{1}{D_{r,\alpha}}\biggl(\frac{5(h+\hat{\alpha}_r)}{3(t-h-\hat{\alpha}_r-1)} + \frac{2(h+\hat{\alpha}_r)}{(t-h-\hat{\alpha}_r-1)^2} \\&\qquad\qquad+ \frac{1158(h+\hat{\alpha}_r)}{8(t-h-\hat{\alpha}_r-1)^2\log(t-h-\hat{\alpha}_r-1)}\biggr)\biggr)
\end{align*}
where
\begin{align*}
    B_{1,\alpha,r} &= \frac{c_{1,r}}{ D_{r,\alpha}}\\
    B_{2,\alpha,r} &= \frac{c_{2,r}}{ \pi  D_{r,\alpha}}\\
    B_{3,\alpha,r} &= \frac{c_{3,r}}{  \pi  D_{r,\alpha}} \\
    B_{4,\alpha,r} &= \frac{ 2\Phi(1) + c_{5,r}}{ D_{r,\alpha}}.
\end{align*}
See Table~\ref{tab:rectangularjensen} for some optimized values of $B_{1,\alpha,r}$, $B_{2,\alpha,r}$, $B_{3,\alpha,r}$, and $B_{4,\alpha,r}$.

If in addition $t>10^{12}$ and $h<t^{2/3}$, then the $\Ostar$ term can be replaced with
\[
\Ostar\!\left(\frac{1}{D_{r,\alpha}}\left(\frac{3.34}{t^{1/3}} + \frac{4.01}{t^{4/3}} + \frac{291}{t^{4/3}\log|t|}\right)\right).
\]
\end{theorem}
\begin{proof}
The result follows from 
Proposition~\ref{prop:jensenrectangle} by applying the bounds from Propositions~\ref{prop:jensen-easy}, \ref{prop:outside2} and \ref{prop:outside3}.
\end{proof}

\begin{table}
\caption{Values for Theorem~\ref{thm:rectangularjensen}.
Values, including $r$, are rounded up.}\label{tab:rectangularjensen}
\setlength{\tabcolsep}{3pt}\footnotesize
\begin{tabular}{|cc|rrrrr|r|}
\hline
$\alpha$&$r$&$B_{1,\alpha,r}$&$B_{2,\alpha,r}$&$B_{3,\alpha,r}$&$B_{4,\alpha,r}$&$\hat{\alpha}_r$&$r/\alpha$ \\
\hline
$ 1/7 $&$0.5647100$&$0.2457053$&$0.3900919$&$1.450119$&$6.518765$&$0.5463417$&$3.952970$\\
$ 1/8 $&$0.5185446$&$0.2294753$&$0.4239720$&$1.514935$&$6.977388$&$0.5032529$&$4.148357$\\
$ 1/9 $&$0.5030821$&$0.2172051$&$0.4340594$&$1.446353$&$6.922300$&$0.4906586$&$4.527739$\\
$ 1/10 $&$0.4821212$&$0.2079206$&$0.4778655$&$1.331336$&$7.058488$&$0.4716364$&$4.821212$\\
$ 1/11 $&$0.3820330$&$0.1991813$&$0.8624038$&$0.9519504$&$9.637060$&$0.3710589$&$4.202362$\\
$ 1/12 $&$0.3390065$&$0.1907470$&$1.076388$&$0.7557944$&$11.128858$&$0.3286046$&$4.068078$\\
$ 1/13 $&$0.3165097$&$0.1832274$&$1.186113$&$0.6552031$&$11.906782$&$0.3070199$&$4.114626$\\
$ 1/14 $&$0.3035626$&$0.1767236$&$1.236003$&$0.6023857$&$12.253958$&$0.2950393$&$4.249876$\\
$ 1/15 $&$0.2957571$&$0.1711536$&$1.252382$&$0.5730817$&$12.349950$&$0.2881455$&$4.436356$\\
$ 1/16 $&$0.2909922$&$0.1663861$&$1.250289$&$0.5551999$&$12.302409$&$0.2842010$&$4.655874$\\
$ 1/32 $&$0.1701799$&$0.1273443$&$2.016620$&$0.8785824$&$20.388630$&$0.1672861$&$5.445757$\\
$ 1/64 $&$0.09924636$&$0.1023306$&$3.292282$&$1.434352$&$33.994587$&$0.09800867$&$6.351767$\\
$ 1/128 $&$0.05717608$&$0.08524602$&$5.501905$&$2.397021$&$57.665580$&$0.05663982$&$7.318538$\\
$ 1/256 $&$0.03250649$&$0.07294429$&$9.396536$&$4.093798$&$99.501950$&$0.03227093$&$8.321661$\\
$ 1/512 $&$0.01825477$&$0.06370419$&$16.349070$&$7.122815$&$174.312148$&$0.01814999$&$9.346441$\\
$ 1/1024 $&$0.01014060$&$0.05652525$&$28.889451$&$12.586295$&$309.383987$&$0.01009347$&$10.383974$\\
$ 1/2048 $&$0.005580532$&$0.05079351$&$51.707199$&$22.527325$&$555.296995$&$0.005559130$&$11.428930$\\
\hline
\end{tabular}
\end{table}
\begin{table}
\caption{Values for $c_{1,r}$, $c_{2,r}$, $c_{3,r}$ as defined in Proposition~\ref{prop:jensen-easy}, and $c_{5,r}$ as defined in Proposition~\ref{prop:outside3} for the unrounded $r$ values from Table~\ref{tab:rectangularjensen}. All values are rounded up.}\label{tab:crvalues3}
\begin{tabular}{|c|cccc|}
\hline
r  & $c_{1,r}$& $c_{2,r}$& $c_{3,r}$& $c_{5,r}$\\
\hline
$0.5647100$&$0.1761601$&$0.8786375$&$3.266227$&$1.078532$\\
$0.5185446$&$0.1548203$&$0.8986260$&$3.210968$&$1.112301$\\
$0.5030821$&$0.1480773$&$0.9296458$&$3.097724$&$1.124066$\\
$0.4821212$&$0.1394940$&$1.007195$&$2.806049$&$1.140402$\\
$0.3820329$&$0.09963112$&$1.355212$&$1.495928$&$1.225349$\\
$0.3390065$&$0.08331889$&$1.477082$&$1.037145$&$1.265982$\\
$0.3165097$&$0.07515015$&$1.528325$&$0.8442389$&$1.288392$\\
$0.3035626$&$0.07062081$&$1.551699$&$0.7562453$&$1.301685$\\
$0.2957571$&$0.06797647$&$1.562641$&$0.7150543$&$1.309847$\\
$0.2909922$&$0.06640649$&$1.567667$&$0.6961340$&$1.314886$\\
$0.1701799$&$0.03157371$&$1.570797$&$0.6843504$&$1.460012$\\
$0.09924635$&$0.01554098$&$1.570797$&$0.6843504$&$1.567632$\\
$0.05717607$&$0.007746955$&$1.570797$&$0.6843504$&$1.645371$\\
$0.03250649$&$0.003881447$&$1.570797$&$0.6843504$&$1.699469$\\
$0.01825477$&$0.001948252$&$1.570797$&$0.6843504$&$1.735811$\\
$0.01014060$&$9.783026\cdot 10^{-4}$&$1.570797$&$0.6843504$&$1.759479$\\
$0.005580530$&$4.911646\cdot 10^{-4}$&$1.570797$&$0.6843504$&$1.774490$\\
\hline
\end{tabular}
\end{table}

\section{Littlewood Zero Detection}\label{sec:littlewood}

An alternative to the use of Jensen's formula is Littlewood's zero counting strategy. 
This is often used to count zeros on large intervals (see for example \cite{KadiriLumleyNg2018,Farzanfard2024,Chourasiya2024,ChourasiyaSimonic2025,Simonic2019}).
This approach can be adapted to the current setting.
In other applications one often mollifies the $L$-function whose zeros are being counted. However, for short intervals this isn't needed.

The core idea of the method arises from the Littlewood's zero detector (see \cite[Section~9.9]{Titchmarsh1986})
\begin{equation}\label{eq:zerodensityintegral}
\begin{aligned}
\int_{\sigma_0}^{\sigma_1} N(\sigma,T_1,T_2),d\sigma &= \frac{1}{2\pi}\Biggl(
 \int_{T_1}^{T_2} \log|f(\sigma_0 +it)|\,dt
- \int_{T_1}^{T_2} \log|f(\sigma_1 +it)|\,dt\\
& + \int_{\sigma_0}^{\sigma_1} \arg f( u+iT_2)\,du
 - \int_{\sigma_0}^{\sigma_1} \arg f( u+iT_1)\,du \Biggr).
 \end{aligned}
 \end{equation}
In the above $\arg f$ is always defined by continuous variation along a vertical line to the right of any zeros (that is, to the right of $1$) followed by passing along a horizontal line, provided that horizontal line does not contain a zero of $f$. One handles the situation where there is a zero on the horizontal line by defining, compatibly with the halving convention for $N(\sigma,T_1,T_2)$ 
\[ \lim_{\epsilon\to0} \frac{1}{2}\left(\arg f(T_2 + \epsilon)+\arg f( T_2 - \epsilon)\right) \] 
so that \eqref{eq:zerodensityintegral} still holds.
 
The standard way to make use of this is to translate it to a bound on $N(T-h,T+h,\alpha)$ through:
\begin{equation}\label{eq:simplelittlewoodzerodensity}
    \begin{aligned} N(T&-h, T+h,\alpha)  \\&
    \leq\frac{1}{2\pi(1-\sigma_0-\alpha)} \Biggl(  \int_{T-h}^{T+h} \log|\zeta(\sigma_0 +it)|\,dt
- \int_{T-h}^{T+h} \log|\zeta(\sigma_1 +it)|\,dt\\
& \qquad\qquad+ \int_{\sigma_0}^{\sigma_1} \arg \zeta( u+i(T+h))\,du
 - \int_{\sigma_0}^{\sigma_1} \arg \zeta( u+i(T-h))\,du \Biggr).
    \end{aligned}
\end{equation}
However, this is not necessarily the best use of Equation \eqref{eq:zerodensityintegral}, see Remark~\ref{rem:useintegralbound}.

We shall bound $\int_{T-h}^{T+h} \log|\zeta(\sigma_0 +it)|\,dt$ trivially.
We will use the approach of Section~\ref{sec:refinement-integration} to bound $-\int_{T-h}^{T+h} \log|\zeta(\sigma_1 +it)|\,dt$ and in Section~\ref{sec:argintegrals} we shall provide an improved method of bounding $\int_{\sigma_1}^{\sigma_0} \arg \zeta(u+i(T+h)) du$ and $\int_{\sigma_1}^{\sigma_0} \arg \zeta(u+i(T-h)) du$.

The improvements from Section~\ref{sec:argintegrals} are crucial to our method, but also provide a strategy to improve the secondary error terms of essentially every recent paper that employs this strategy as they all rely on a Lemma of Kadiri, Lumley and Ng (See \cite[Lemma~4.12]{KadiriLumleyNg2018}) which bounds the integral by bounding the argument at $1/2+iT$ rather than by bounding the integral itself.

\begin{lemma}\label{lemma:littlewood-mainterm}
With notation as defined in Equations \eqref{eq:sigmak}, \eqref{eq:mk}, \eqref{eq:mkp}, \eqref{eq:mkpp} suppose that $\sigma_k\leq\sigma < \sigma_{k+1}$, that $0<h$, and that $T-h>12$,
then
\begin{align*}  \frac{1}{2\pi}\int_{T-h}^{T+h} \log|\zeta(\sigma +it)|\,dt &< h\Biggl(\frac{m_k\sigma + b_k}{\pi}\log|T| + \frac{m_k'\sigma + b_k'}{\pi}\log\log|T| \\&\qquad + \frac{m_k''\sigma + b_k''}{\pi}\Biggr) +  \Ostar\!\left(\frac{h}{\pi(T-h)^2}\left(2+\frac{1158}{8\log(T-h)}\right)\right) .
\end{align*}
\end{lemma}
\begin{proof}
This follows immediately from the validity of the bound on $\log|\zeta(\sigma +it)|$.

To obtain the given $\Ostar$ term one should actually integrate the $\log|t|$ and $\log\log|t|$ terms along with the $\Ostar$ error terms from Proposition~\ref{prop:interpolation}, rather than taking the naive bound, and eventually consider the Taylor expansions of various terms as a function of $\frac{h}{t}$.
\end{proof}

\begin{lemma}\label{lemma:littlewood-outerlog}
Suppose that $0\leq h < T$ and $1 \leq \sigma$ then
\[ -\frac{1}{2\pi}\int_{T-h}^{T+h} \log|\zeta(\sigma +it)|\,dt \leq \frac{\Phi(\sigma)}{\pi}
\]
where $\Phi(\sigma)$ is the anti-derivative of $-\log(\zeta)$ as defined in \eqref{eq:phis}.
\end{lemma}
 \begin{proof}
This follows from the fundamental theorem of calculus and the absolute convergence of the series defining $\Phi$.
\end{proof}

We shall denote by 
\[ \Delta \arg f(x)|_a^b = \arg(f(b)) - \arg(f(a)) \]
the change in $\arg f$ along the line from $a$ to $b$. We shall only use this when there should be no ambiguity about the path.
\begin{lemma}\label{lem:littlewood-argsetup}
   Suppose that $1< \sigma_1$ and $\sigma_0\leq\sigma_1$ then
    \[ |\arg(\zeta(\sigma_1 +iT+ih)) - \arg(\zeta(\sigma_1 +iT-ih))| \leq 2\log(\zeta(\sigma_1)) \]
    and hence
    \begin{align*}
    \frac{1}{2\pi}\Biggl(\int_{\sigma_0}^{\sigma_1} &
     \arg \zeta(u+i(T+h))\,du - \int_{\sigma_0}^{\sigma_1} \arg \zeta(u+i(T-h))\,du \Biggr)\\
    &\leq \frac{\sigma_1-\sigma_0}{\pi}\log(\zeta(\sigma_1)) +
    \frac{1}{2\pi}\int_{\sigma_0}^{\sigma_1} \Delta\arg\zeta(iT+ih+x)|_{\sigma_1}^{u}\,du \\&
     \qquad- \frac{1}{2\pi}\int_{\sigma_0}^{\sigma_1} \Delta\arg\zeta(iT-ih+x)|_{\sigma_1}^{u}\,du.
    \end{align*}
\end{lemma}
\begin{proof}
    We have $|\arg(\zeta(\sigma +iT+ih) | < |\log(\zeta(\sigma +iT+ih) | < \log(\zeta(\sigma))$.
\end{proof}

In the next section we discuss how to bound the integral of the change in argument of a function.

\subsection{Bounding Integrals of $\Delta \arg f(x)$}\label{sec:argintegrals}

Backlund's method has become a standard method to bound the argument of an $L$-function in the critical strip \cite{Backlund1916,Trudgian2014,HasanalizadeShenWong2022,BellottiWongFiori2025}. 

However, in the context of the use of Littlewood's zero detection method as above we do not need to bound the argument per-se but rather we can write
\[ \int_{\sigma_1}^{\sigma_0} \arg f(\sigma+iT)\,d\sigma = (\sigma_0- \sigma_1)\arg f(\sigma_1+iT) + \int_{\sigma_1}^{\sigma_0} \Delta\arg f( x + i T)|_{\sigma_1}^u\,du. \]
We provide here a re-interpretation of Backlund's trick which is tuned to bounding directly
\[ \int_{\sigma_1}^{\sigma_0} \Delta\arg f( x + i T)|_{\sigma_1}^u \,du. \]
For $L$-functions this method is particularly well adapted to the situation where $\sigma_1 > 1/2$.
The advantage of our approach is that we don't need to consider $f$ in the region to the left of $\sigma_1$, for instance on the left half of the critical strip where its values are known to be large by virtue of the functional equation and the $\Gamma$ factors.

A standard strategy to bound the change in the argument of a function is to count the number of zeros of $\Re(f)$ along the path.
This can be amplified by considering 
\[ F_N(u) = \frac{1}{2}\left(f(u+c+iT)^N + f(u+c-iT)^N\right) \]
which has $\Re(f(c+u+iT)^N) = F_N(u)$.

The following Lemma provides the core of the new idea here, which is to recognize we wish to integrate this number of zeros, rather than focus on bounding the number of zeros.
\begin{lemma}\label{lem:integralofarg}
Fix $T>0$, $N>0$, $\sigma_0<\sigma_1$ and $c>(\sigma_0+\sigma_1)/2$ then
\[ \Bigl|\int_{\sigma_1}^{\sigma_0} \Delta\arg f( x + i T)|_{\sigma_1}^u\,du\Bigr| \leq \frac{\pi}{N}\left( (\sigma_1-\sigma_0) + C_{c,\sigma_0,\sigma_1}\int_0^{c-\sigma_0} \frac{n_{F_N}(u)}{u}\,du \right) \]
where
\[  F_N(u) = \frac{1}{2}\left(f(u+c+iT)^N + f(u+c-iT)^N)\right) \]
and $n_{F_N}(u)$ denotes the number of complex zeros for the function $F_N$ in a ball of radius $u$ and
\[ C_{c,\sigma_0,\sigma_1}= \sup_{\sigma\in[\sigma_0,\sigma_1]}\left(\frac{\sigma - \sigma_0}{\log(c-\sigma_0)-\log|\sigma -c|} \right). \]
\end{lemma}
\begin{proof}
First, note that we can reduce to the case of having no zeros on the horizontal line through the halving convention described after \eqref{eq:zerodensityintegral} using uniform continuity. 

Next, we note that both the left and right hand integrals can be bounded in terms of zeros of $F_N$.

Let $S_L$ denote the zeros of $\Re(f)$ between $c$ and $\sigma_0$ and let $S_R$ denote the zeros between $c$ and $\sigma_1$.
We have
\[ \left|\int_{\sigma_1}^{\sigma_0} \Delta\arg f( x + i T)|_{\sigma_1}^u\,du \right| \leq \frac{\pi}{N} \left((\sigma_1-\sigma_0)+ \sum_{\rho\in S_L\cup S_R} (\Re(\rho) -\sigma_0) \right). \]
This follows from the observation that $\arg f$ can be bounded by a step function which increases by $\frac{\pi}{N}$ every time $\Re(f(u)^N) = 0$, that is, precisely when $F_N(u) = 0$, and which starts at the value $\frac{\pi}{N}$ at $u=\sigma_1$.
The formula on the right is precisely the integral of this step function.

On the other hand, the right hand integral can be rewritten as
\begin{align*}
\int_0^{c-\sigma_0} \frac{n_{F_N}(u)}{u}\,du
& = \sum_{\substack{F_N(\rho)=0\\ |\rho|<c-\sigma_0}} \log\left(\frac{c-\sigma_0}{|\rho|}\right)
\\ &\geq \sum_{\rho\in S_L} \log\left(\frac{c-\sigma_0}{c-\Re(\rho)}\right)
+ \sum_{\rho\in S_R}   \log\left(\frac{c-\sigma_0}{\Re(\rho)-c}\right),
\end{align*}
since every real zero is a complex zero and the dropped terms are nonnegative.
As we have defined $C_{c,\sigma_0,\sigma_1}$ to be the maximum possible ratio between the weight for a zero in the respective sums, this gives the result.
\end{proof}

Depending on where $c$ is in the interval $((\sigma_0+\sigma_1)/2, \infty)$ the supremum defining $C_{c,\sigma_0,\sigma_1}$ will occur at different places.

 We define $\eta \in(0,1)$ to be the unique solution to 
\begin{equation}\label{eq:eta}
1+\eta + \log(\eta)=0.
\end{equation}
Note that $\eta = 0.27846454276\ldots$

\begin{lemma}\label{lem:supremumforarg}
We have
\[
\sup_{\sigma\in[\sigma_0,c]}\left(\frac{\sigma  - \sigma_0}{\log(c-\sigma_0)-\log|\sigma -c|} \right) \leq c-\sigma_0
\]
and, provided $\frac{\sigma_0+\sigma_1}{2} < c$,
\[
\sup_{\sigma\in[c,\sigma_1]}\left(\frac{\sigma  - \sigma_0}{\log(c-\sigma_0)-\log|\sigma -c|} \right) \leq \frac{\sigma_1-\sigma_0}{\log(c-\sigma_0)-\log|\sigma_1-c|}.
\]
Moreover, if $\sigma_1-c \leq \eta(c-\sigma_0)$ where $\eta$ is as defined in \eqref{eq:eta}
then
\[
 \sup_{\sigma\in[\sigma_0,\sigma_1]}\left(\frac{\sigma  - \sigma_0}{\log(c-\sigma_0)-\log|\sigma -c|} \right) \leq c-\sigma_0
\]
\end{lemma}
\begin{proof}
We are looking for extreme values of
\[ \frac{\sigma - \sigma_0}{\log(c-\sigma_0)-\log(|\sigma-c|)}.\]
differentiating with respect to $\sigma$ gives
\[ \frac{\log(c-\sigma_0)-\log(|\sigma-c|)  + \frac{\sigma-\sigma_0}{\sigma-c}}{(\log(c-\sigma_0)-\log(|\sigma-c|))^2}. \]
To the left of $c$, there is potentially a critical point which occurs at $\hat{\sigma}$ a solution to
\[  (\log(c-\sigma_0)-\log(c-\hat{\sigma}))(c-\hat{\sigma}) = (\hat{\sigma}-\sigma_0)  \]
substituting this back into 
\[ \frac{\sigma - \sigma_0}{\log(c-\sigma_0)-\log(|\sigma-c|)}.\]
tells us the critical value will be
\[(c-\hat{\sigma}) \leq c-\sigma_0 .\]
As this is in fact the value of the limit at the left end point of the interval, and since the ratio tends to $0$ as $\sigma\rightarrow c$ this bounds the supremum, which gives us the first claim.

On the right of $c$ we see that the function is increasing, this gives the second claim.

For the final claim, if $(\sigma_1 - c) \leq \eta(c-\sigma_0)$ then
\begin{align*}
  \frac{\sigma_1 -\sigma_0}{
\log(c-\sigma_0) - \log(\sigma_1-c)}  
&=
\frac{(c-\sigma_0) + (\sigma_1 - c)}{
\log(c-\sigma_0) - \log(\sigma_1-c)}\\
&\leq  \frac{ (1+\eta)(c-\sigma_0)}{-\log(\eta)}
= c-\sigma_0
\end{align*}
so that the extreme value from the right is not larger than $c-\sigma_0$.
\end{proof}

\begin{remark}
The above optimization will allow us to consider having $c=1$ even if $\sigma_1>1$. However, it should be noted that in applications of Backlund's trick to bounding $N(\sigma,T)$ there can be lower bounds on $\sigma_1$ arising from the need to bound $\arg(f(\sigma_1+iT))$ or to bound an integral outside the critical strip. 

These lower bounds are often of the form $\sigma_1>1+\frac{k\log\log(T)}{\log(T)}$, (see \cite{Farzanfard2024}). Thus, when considering $N(\sigma,T)$ for very small $\sigma$ extra care must be taken. 
\end{remark}

\begin{theorem}\label{thm:arg-integrals}
Assume $c\in [\sigma_0,\sigma_1]$ satisfies $\sigma_1 - c \leq \eta(c-\sigma_0) $ where $\eta$ is as defined in Equation \eqref{eq:eta}.
Let $r=c-\sigma_0$.
Suppose $f$ is analytic on the disc of radius $r$ centered at $c+iT$, that $\overline{f(s)} = f(\overline{s})$ and that $f$ has no zeros on the line connecting $\sigma_0+iT$ to $\sigma_1+iT$.
 Suppose $F_{c,r}(\theta)$ is an even, integrable real valued function such that
\[ \max\left(\log|f(c+re^{i\theta}+iT)|,\log|f(c+re^{i\theta}-iT)|\right) \leq F_{c,r}(\theta). \]
Then 
\begin{align*} \frac{1}{2\pi}\Biggl| \int_{\sigma_1}^{\sigma_0} \arg f(\sigma+iT)&\,d\sigma - (\sigma_0 - \sigma_1)\arg f(\sigma_1+iT)\Biggr| \\
 &\leq
\frac{c - \sigma_0}{2}\left(
-  \log|f(c+iT)| 
+ \frac{1}{\pi}\int_{0}^{\pi} F_{c,r}(\theta)\,d\theta \right).
\end{align*}
In the above $\arg f$ denotes a continuous argument of $f$ along the line segment connecting $\sigma_0+iT$ to $\sigma_1+iT$.
\end{theorem}
\begin{proof}
First, note that we can reduce to the case of having no zeros on the horizontal line through the halving convention described after \eqref{eq:zerodensityintegral} using uniform continuity. 

Now, beginning from
\begin{align*}
 \frac{1}{2\pi}\left|\int_{\sigma_1}^{\sigma_0} \arg f(\sigma+iT)\,d\sigma - (\sigma_0-\sigma_1)\arg f(\sigma_1+iT)\right|=
\frac{1}{2\pi}\left|\int_{\sigma_1}^{\sigma_0} \Delta\arg f( x + i T)|_{\sigma_1}^u\,du\right|
\end{align*}
Next, we note (as in \cite[Lemma~3.3]{HasanalizadeShenWong2022}) that there exists an infinite collection of $N$ such that $F_N(0)\neq 0$, and for such $N$,
by applying Lemmas~\ref{lem:integralofarg}, \ref{lem:supremumforarg} and then Jensen's Theorem, we have
\begin{align*}
& \frac{1}{2\pi}\left|\int_{\sigma_1}^{\sigma_0} \arg f(\sigma+iT)\,d\sigma - (\sigma_0-\sigma_1)\arg f(\sigma_1+iT)\right|\\
&\leq \frac{\sigma_1-\sigma_0}{2N} + \frac{c-\sigma_0}{2N}
     \int_{0}^{r} \frac{ n_{F_N}(u)}{u}\,du\\
&=\frac{\sigma_1-\sigma_0}{2N} + \frac{c-\sigma_0}{2}\Biggl(
-\frac{1}{N}\log|F_N(0)|+   \frac{1}{2\pi N}\int_{-\pi/2}^{3\pi/2}\log|F_N(re^{i\theta})|\,d\theta \Biggr).\\
&=\frac{\sigma_1-\sigma_0}{2N} + \frac{c-\sigma_0}{2}\Biggl(
-\frac{1}{N}\log|F_N(0)|
\\&\qquad +   \frac{1}{2\pi N}\int_{0}^{\pi}\left(\log|F_N(re^{i\theta})|+\log|F_N(re^{-i\theta})|\right)\,d\theta \Biggr).\\
&=\frac{\sigma_1-\sigma_0}{2N} + \frac{c-\sigma_0}{2}\Biggl(
-\frac{1}{N}\log|F_N(0)| +   \frac{1}{\pi N}\int_{0}^{\pi}\log|F_N(re^{i\theta})|\,d\theta \Biggr).
\end{align*}
Note that the last line is using $\overline{F_N(re^{i\theta})} = F_N(re^{-i\theta})$ so that they have the same absolute values.

We may apply Backlund's trick (as in \cite[Lemma~3.3]{HasanalizadeShenWong2022}) to get a sequence $N_m$ such that
\[ \lim_{m\rightarrow\infty} \log|F_{N_m}(c + iT)|  = N_m \log|f(c+iT)|. \]
Moreover, we have
\begin{align*}
  &\frac{1}{N}\log \frac{1}{2}\left| f(c + iT + re^{i\theta})^N + f(\sigma + iT + re^{i\theta})^N\right|\\ &\qquad \leq 
   \frac{1}{N}\log \max\left(\left|f(c + iT + re^{i\theta})\right|^N, \left|f(\sigma - iT + re^{i\theta})\right|^N\right) \\ 
  &\qquad = \log \max\left(\left|f(c + iT + re^{i\theta})\right|, \left|f(c + iT + re^{-i\theta})\right|\right) \\
  &\qquad \leq F_{c,r}(\theta). 
\end{align*}
Taking $m$ to infinity we obtain
\begin{align*}
&\frac{1}{2\pi}\left|\int_{\sigma_1}^{\sigma_0} \arg f(\sigma+iT)\,d\sigma
- (\sigma_0-\sigma_1)\arg f(\sigma_1+iT)\right|\\&\qquad\leq
\frac{c-\sigma_0}{2}\left(-\log|f(c + iT)|
+  \frac{1}{\pi}\int_{0}^\pi F_{c,r}(\theta)\,d\theta\right) .\qedhere
\end{align*}
\end{proof}

\subsection{Application to $\zeta$ for zero density for short intervals}\label{sec:littlewoodshortapp}

\begin{theorem}\label{thm:littlewoodshort}
  With notation as defined in Equations \eqref{eq:sigmak}, \eqref{eq:mk}, \eqref{eq:mkp}, \eqref{eq:mkpp}, and $\eqref{eq:eta}$.
  Suppose $t-h>13$, $0<h$ and $0<\alpha<r<1$.
  Suppose that $\sigma_k\leq(1-r) < \sigma_{k+1}$
  then
\begin{align*} N(t-h,t+h,\alpha) &<
\frac{2A_{1,\alpha,r}h + 2A_{4,\alpha,r}}{\pi}\log|t| +(2A_{2,\alpha,r}h  +  2A_{5,\alpha,r})\log\log|t| \\&\quad+2A_{3,\alpha,r}h
 + A_{6,\alpha,r}
\\&\quad+\Ostar\!\Biggl(\frac{1}{r-\alpha} \Biggl(\frac{h}{\pi(t-h)^2}\left(2+\frac{1158}{8\log(t-h)}\right) 
\\&\qquad+ \frac{r}{4}\biggl(
 \frac{2r}{3(t+h-r)} + \frac{r}{(t+h-r)\log(t+h) - r} + \frac{2}{(t+h-r)^2}
 \\&\qquad\qquad+\frac{1158}{8(t+h-r)^2\log(t+h-r)} +\frac{2r}{3(t-h-r)} 
 \\&\qquad\qquad+ \frac{r}{(t-h-r)\log(t-h) - r} + \frac{2}{(t-h-r)^2}
 \\&\qquad\qquad+\frac{1158}{8(t-h-r)^2\log(t-h-r)}
\biggr)\Biggr)\Biggr)
\end{align*}
where $A_{1,\alpha,r},\ldots,A_{6,\alpha,r}$ are defined by: 
\begin{align*}
A_{1,\alpha,r} &= \frac{m_k (1-r) + b_k}{2(r-\alpha)}\\
A_{2,\alpha,r} &= \frac{m_k' (1-r) + b_k'}{2\pi(r-\alpha)} \\
A_{3,\alpha,r} &= \frac{m_k''(1-r) + b_k''}{2\pi(r-\alpha)} \\
A_{4,\alpha,r} &= \frac{c_{1,r}r}{(r-\alpha)} \\
A_{5,\alpha,r} &= \frac{c_{2,r}r + \pi r}{\pi(r-\alpha)}\\
A_{6,\alpha,r} &= \frac{c_{3,r}r+  c_{4,r}\pi r + \pi\log(29.388)r+ \Phi(1+\eta r)  + (1+\eta)r\log(\zeta(1+\eta r))}{\pi(r-\alpha)}.
\end{align*}
Calculated values for these constants are provided in Table~\ref{tab:littlewoodshort}.

If in addition $t>10^{12}$ and  $h<t^{2/3}$, then the $\Ostar$ term can be replaced with
\[
\Ostar\!\left(\frac{1}{t(r-\alpha)}\left(0.334 + \frac{0.502}{\log|t|} + \frac{1.001}{t} + \frac{72.48}{t\log|t|} + \frac{0.64}{t^{1/3}} + \frac{46.3}{t^{1/3}\log|t|}\right)\right) .
\]
%

\end{theorem}
\begin{proof}
This follows from Equation \eqref{eq:simplelittlewoodzerodensity} with $\sigma_0=1-r$ and $\sigma_1=1+\eta r$.
We use Lemma~\ref{lemma:littlewood-mainterm} to bound the first integral which gives the $A_{1,\alpha,r}$, $A_{2,\alpha,r}$, and $A_{3,\alpha,r}$ terms.
We use Lemma~\ref{lemma:littlewood-outerlog} to bound the integral outside the critical strip, this contributes $\Phi(1+\eta r)/\pi$ to $A_{6,\alpha,r} $.
We use Lemma~\ref{lem:littlewood-argsetup} to convert the $\arg\zeta$ integrals into $\Delta\arg\zeta$ integrals.
Using 
\[ \arg\zeta(1+\eta r + i(t \pm h)) \leq |\log(\zeta(1+\eta r + i(t \pm h)))| < \log(\zeta(1+\eta r)) \]
 contributes $(1+\eta)r\log(\zeta(1+\eta r))/\pi$ to $A_{6,\alpha,r}$.

Now for each of the two $\Delta\arg\zeta$ integrals we take $c=1$ and define $F_{c,r}(\theta)$ to be  $ \log(\zeta(r\cos(\theta))) $ when $ 0 \leq \theta \leq \pi/2$ and defined as in the proof of Proposition~\ref{prop:jensen-easy}.
The combined contribution from these two integrals is then precisely $1-\sigma_0=r$ times the bound obtained in Theorem~\ref{thm:circularregions1}. This explains the additional terms in $A_{4,\alpha,r}$, $A_{5,\alpha,r}$, and $A_{6,\alpha,r}$. 

Finally, the $\Ostar$ term tracks the errors that accumulate from the inputs above that carried their own $\Ostar$ error.
\end{proof}

\begin{remark}
    In contrast to optimizing $r$ for $B_{1,\alpha,r}$ the strategy for $A_{1,\alpha,r}$ is relatively straightforward as it is essentially defined by a piecewise fractional linear map:
    \[   \frac{-m_k r + m_k+b_k}{r-\alpha} \]
    so that on any interval its derivative  
    \[   \frac{  m_k\alpha + m_k+b_k}{(r-\alpha)^2}\]
    doesn't change sign. As such, the optimal is at an endpoint.
\end{remark}

\begin{remark}\label{rem:useintegralbound}
    It is worth noting that in some applications the bound in Equation \eqref{eq:zerodensityintegral} can be used directly rather than needing to translate to a bound through Theorem~\ref{thm:littlewoodshort}. This is because in applications we are often interested in bounding
    \[ \int_{\sigma_1}^{\sigma_2} f(\sigma)N(\sigma,T_1,T_2)d\sigma \]
    and integration by parts allows, for at least some of the terms, to be bounded by a direct use of equation \eqref{eq:zerodensityintegral}. This would typically result in a better final bound than using the bound from Theorem~\ref{thm:littlewoodshort} immediately.
\end{remark}

\begin{table}
\caption{Values for Theorem~\ref{thm:littlewoodshort}.
All values, including $r$, are rounded up.}\label{tab:littlewoodshort}

\setlength{\tabcolsep}{2pt}\footnotesize
\resizebox{\textwidth}{!}{%
\begin{tabular}{|cc|rrrrrr|r|}
\hline
$\alpha$&$r$&$A_{1,\alpha,r}$&$A_{2,\alpha,r}$&$A_{3,\alpha,r}$&$A_{4,\alpha,r}$&$A_{5,\alpha,r}$&$A_{6,\alpha,r}$&$r/\alpha$ \\
\hline
$ 1/6 $&$0.5000000$&$0.2469513$&$0$&$2.005451$&$0.2201801$&$1.948504$&$10.269668$&$3.000000$\\
$ 1/7 $&$0.5000000$&$0.2304879$&$0$&$1.871754$&$0.2055015$&$1.818604$&$9.585023$&$3.500000$\\
$ 1/8 $&$0.5000000$&$0.2195122$&$0$&$1.782623$&$0.1957157$&$1.732003$&$9.128594$&$4.000000$\\
$ 1/9 $&$0.2857143$&$0.2045455$&$0.9115238$&$0.3971245$&$0.1059233$&$2.454546$&$12.040681$&$2.571429$\\
$ 1/10 $&$0.2857143$&$0.1923077$&$0.8569882$&$0.3733649$&$0.09958597$&$2.307693$&$11.320299$&$2.857143$\\
$ 1/11 $&$0.2857143$&$0.1833334$&$0.8169954$&$0.3559412$&$0.09493862$&$2.200000$&$10.792018$&$3.142858$\\
$ 1/12 $&$0.2857143$&$0.1764706$&$0.7864127$&$0.3426172$&$0.09138477$&$2.117648$&$10.388039$&$3.428572$\\
$ 1/13 $&$0.2857143$&$0.1710527$&$0.7622685$&$0.3320983$&$0.08857910$&$2.052632$&$10.069108$&$3.714286$\\
$ 1/14 $&$0.2857143$&$0.1666667$&$0.7427231$&$0.3235829$&$0.08630784$&$2.000000$&$9.810926$&$4.000000$\\
$ 1/15 $&$0.2857143$&$0.1630435$&$0.7265770$&$0.3165485$&$0.08443158$&$1.956522$&$9.597645$&$4.285715$\\
$ 1/16 $&$0.2857143$&$0.1600000$&$0.7130142$&$0.3106396$&$0.08285553$&$1.920001$&$9.418489$&$4.571429$\\
$ 1/32 $&$0.1666667$&$0.1230770$&$1.175299$&$0.5120433$&$0.03779972$&$1.846154$&$11.325849$&$5.333334$\\
$ 1/64 $&$0.09677420$&$0.09937889$&$1.961264$&$0.8544655$&$0.01794826$&$1.788820$&$14.441878$&$6.193549$\\
$ 1/128 $&$0.05555556$&$0.08311689$&$3.333573$&$1.452341$&$0.008701209$&$1.745455$&$19.713406$&$7.111112$\\
$ 1/256 $&$0.03149607$&$0.07134895$&$5.768613$&$2.513217$&$0.004266277$&$1.712375$&$28.844815$&$8.062993$\\
$ 1/512 $&$0.01764706$&$0.06246950$&$10.141176$&$4.418216$&$0.002105235$&$1.686677$&$44.972432$&$9.035295$\\
$ 1/1024 $&$0.009784736$&$0.05554351$&$18.069007$&$7.872142$&$0.001042787$&$1.666306$&$73.905261$&$10.019570$\\
$ 1/2048 $&$0.005376345$&$0.04999512$&$32.559922$&$14.185413$&$5.177350\cdot 10^{-4}$&$1.649839$&$126.451532$&$11.010753$\\
\hline
\end{tabular}}
\end{table}

\section{Tables of $C_2$ values}\label{sec:C2tables}

For each $\alpha$ in Table~\ref{tab:CRvalues}'s list with $\alpha<1/6$, and each pair
$(h_0,t_0)$, the value of $r\in(\alpha,1)$ minimizing $C_2$ was located numerically and the
resulting $C_2$ recorded. Since $C_2$ is decreasing in both $h_0$ and $t_0$, each row bounds
the zero-density ratio for \emph{all} $t>t_0$ and $t^{2/3}>h>h_0$. Rows attaining $C_2<1/2$ guarantee that a positive proportion, $(1-2C_2)$, of the zeros in the interval
have real part in $[\alpha,1-\alpha]$; these are the values recorded in Table~\ref{table:chat}.

\begin{table}[p]
\centering\footnotesize
\setlength{\tabcolsep}{3pt}
\caption{Values of $C_2^J(\alpha,r,h_0,t_0)$ for Corollary \ref{cor:main-jensen}, with $r$ chosen to minimise $C_2^J$, at selected $t_0$; only entries with $C_2^J<1/2$ are shown. For each $\alpha$ the rows are for $h_0=100$ and, when a smaller value succeeds, for the smallest $h_0\in\{1,2,5,10\}$ with $C_2^J<1/2$ at some tabulated $t_0$.}\label{tab:C2J}
\begin{tabular}{|cc|cc|c||cc|cc|c|}
\hline
$\alpha$&$r$&$h_0$&$t_0$&$C_2^J$&$\alpha$&$r$&$h_0$&$t_0$&$C_2^J$\\
\hline
$1/7$&$0.5636350$&$100$&$10^{10000}$&$0.4986622$&$1/32$&$0.2013314$&$2$&$10^{1000}$&$0.4865377$\\
$1/8$&$0.5342314$&$100$&$10^{1000}$&$0.4775715$&$1/32$&$0.1675280$&$2$&$10^{10000}$&$0.4082385$\\
$1/8$&$0.5189671$&$100$&$10^{10000}$&$0.4657350$&$1/32$&$0.5517097$&$100$&$10^{50}$&$0.4818041$\\
$1/9$&$0.5000332$&$10$&$10^{10000}$&$0.4872129$&$1/32$&$0.5027604$&$100$&$10^{100}$&$0.4124018$\\
$1/9$&$0.5121283$&$100$&$10^{1000}$&$0.4528088$&$1/32$&$0.2266519$&$100$&$10^{1000}$&$0.2964409$\\
$1/9$&$0.5034450$&$100$&$10^{10000}$&$0.4408710$&$1/32$&$0.1731060$&$100$&$10^{10000}$&$0.2624836$\\
$1/10$&$0.5001816$&$10$&$10^{1000}$&$0.4808442$&$1/64$&$0.1689264$&$2$&$10^{1000}$&$0.4175095$\\
$1/10$&$0.4196975$&$10$&$10^{10000}$&$0.4644748$&$1/64$&$0.1011201$&$2$&$10^{10000}$&$0.3266227$\\
$1/10$&$0.5030004$&$100$&$10^{1000}$&$0.4337894$&$1/64$&$0.5286353$&$100$&$10^{50}$&$0.4594898$\\
$1/10$&$0.4952519$&$100$&$10^{10000}$&$0.4220812$&$1/64$&$0.4999542$&$100$&$10^{100}$&$0.3918022$\\
$1/11$&$0.3444898$&$5$&$10^{10000}$&$0.4893983$&$1/64$&$0.1712617$&$100$&$10^{1000}$&$0.2563107$\\
$1/11$&$0.5002041$&$100$&$10^{1000}$&$0.4189652$&$1/64$&$0.1053448$&$100$&$10^{10000}$&$0.2153473$\\
$1/11$&$0.3919117$&$100$&$10^{10000}$&$0.4049475$&$1/128$&$0.1492711$&$2$&$10^{1000}$&$0.3854116$\\
$1/12$&$0.3663995$&$5$&$10^{1000}$&$0.4975744$&$1/128$&$0.0678735$&$2$&$10^{10000}$&$0.2793925$\\
$1/12$&$0.3199260$&$5$&$10^{10000}$&$0.4666552$&$1/128$&$0.5203543$&$100$&$10^{50}$&$0.4486379$\\
$1/12$&$0.5613468$&$100$&$10^{100}$&$0.4907737$&$1/128$&$0.4238279$&$100$&$10^{100}$&$0.3812776$\\
$1/12$&$0.4330556$&$100$&$10^{1000}$&$0.4064223$&$1/128$&$0.1630512$&$100$&$10^{1000}$&$0.2370474$\\
$1/12$&$0.3456430$&$100$&$10^{10000}$&$0.3882898$&$1/128$&$0.0712326$&$100$&$10^{10000}$&$0.1859247$\\
$1/13$&$0.3408036$&$5$&$10^{1000}$&$0.4799246$&$1/256$&$0.0558222$&$1$&$10^{10000}$&$0.4669492$\\
$1/13$&$0.3050524$&$5$&$10^{10000}$&$0.4472619$&$1/256$&$0.5166588$&$100$&$10^{50}$&$0.4433041$\\
$1/13$&$0.5483231$&$100$&$10^{100}$&$0.4805636$&$1/256$&$0.3934370$&$100$&$10^{100}$&$0.3755101$\\
$1/13$&$0.3792615$&$100$&$10^{1000}$&$0.3941154$&$1/256$&$0.1248878$&$100$&$10^{1000}$&$0.2265346$\\
$1/13$&$0.3211564$&$100$&$10^{10000}$&$0.3733128$&$1/256$&$0.0561011$&$100$&$10^{10000}$&$0.1664993$\\
$1/14$&$0.3235592$&$5$&$10^{1000}$&$0.4644272$&$1/512$&$0.0442182$&$1$&$10^{10000}$&$0.4384521$\\
$1/14$&$0.2959856$&$5$&$10^{10000}$&$0.4308666$&$1/512$&$0.5151511$&$100$&$10^{50}$&$0.4406617$\\
$1/14$&$0.5389238$&$100$&$10^{100}$&$0.4718842$&$1/512$&$0.3805800$&$100$&$10^{100}$&$0.3725238$\\
$1/14$&$0.3482491$&$100$&$10^{1000}$&$0.3826411$&$1/512$&$0.1141511$&$100$&$10^{1000}$&$0.2206397$\\
$1/14$&$0.3066785$&$100$&$10^{10000}$&$0.3602738$&$1/512$&$0.0399347$&$100$&$10^{10000}$&$0.1560074$\\
$1/15$&$0.3113914$&$5$&$10^{1000}$&$0.4509017$&$1/1024$&$0.0378093$&$1$&$10^{10000}$&$0.4217035$\\
$1/15$&$0.2907171$&$5$&$10^{10000}$&$0.4170120$&$1/1024$&$0.5143367$&$100$&$10^{50}$&$0.4393469$\\
$1/15$&$0.5318671$&$100$&$10^{100}$&$0.4644368$&$1/1024$&$0.3746734$&$100$&$10^{100}$&$0.3710083$\\
$1/15$&$0.3287816$&$100$&$10^{1000}$&$0.3722727$&$1/1024$&$0.1101808$&$100$&$10^{1000}$&$0.2176024$\\
$1/15$&$0.2980066$&$100$&$10^{10000}$&$0.3490520$&$1/1024$&$0.0339743$&$100$&$10^{10000}$&$0.1493438$\\
$1/16$&$0.3029161$&$5$&$10^{1000}$&$0.4391156$&$1/2048$&$0.0354201$&$1$&$10^{10000}$&$0.4127839$\\
$1/16$&$0.2876722$&$5$&$10^{10000}$&$0.4052587$&$1/2048$&$0.5139859$&$100$&$10^{50}$&$0.4386911$\\
$1/16$&$0.5264389$&$100$&$10^{100}$&$0.4579894$&$1/2048$&$0.3717163$&$100$&$10^{100}$&$0.3702453$\\
$1/16$&$0.3156734$&$100$&$10^{1000}$&$0.3630387$&$1/2048$&$0.1082799$&$100$&$10^{1000}$&$0.2160693$\\
$1/16$&$0.2926453$&$100$&$10^{10000}$&$0.3394116$&$1/2048$&$0.0324890$&$100$&$10^{10000}$&$0.1458626$\\
\hline
\end{tabular}
\end{table}

\begin{table}[p]
\centering\footnotesize
\setlength{\tabcolsep}{3pt}
\caption{Values of $C_2^L(\alpha,r,h_0,t_0)$ for Corollary \ref{cor:main-littlewood}, with $r$ chosen to minimise $C_2^L$, at selected $t_0$; only entries with $C_2^L<1/2$ are shown. For each $\alpha$ the rows are for $h_0=100$ and, when a smaller value succeeds, for the smallest $h_0\in\{1,2,5,10\}$ with $C_2^L<1/2$ at some tabulated $t_0$.}\label{tab:C2L}
\begin{tabular}{|cc|cc|c||cc|cc|c|}
\hline
$\alpha$&$r$&$h_0$&$t_0$&$C_2^L$&$\alpha$&$r$&$h_0$&$t_0$&$C_2^L$\\
\hline
$1/7$&$0.5000553$&$100$&$10^{1000}$&$0.4740568$&$1/16$&$0.2856104$&$100$&$10^{10000}$&$0.3257779$\\
$1/7$&$0.5000553$&$100$&$10^{10000}$&$0.4685756$&$1/32$&$0.1665465$&$2$&$10^{1000}$&$0.4906269$\\
$1/8$&$0.4998543$&$100$&$10^{1000}$&$0.4514720$&$1/32$&$0.0968092$&$2$&$10^{10000}$&$0.4025658$\\
$1/8$&$0.4998543$&$100$&$10^{10000}$&$0.4462476$&$1/32$&$0.5000232$&$100$&$10^{50}$&$0.4488857$\\
$1/9$&$0.2857554$&$10$&$10^{1000}$&$0.4882517$&$1/32$&$0.2856721$&$100$&$10^{100}$&$0.3921085$\\
$1/9$&$0.2857554$&$10$&$10^{10000}$&$0.4619823$&$1/32$&$0.1665465$&$100$&$10^{1000}$&$0.2756404$\\
$1/9$&$0.5000332$&$100$&$10^{100}$&$0.4852020$&$1/32$&$0.1665465$&$100$&$10^{10000}$&$0.2518982$\\
$1/9$&$0.3669506$&$100$&$10^{1000}$&$0.4349491$&$1/64$&$0.0968954$&$2$&$10^{1000}$&$0.4205741$\\
$1/9$&$0.2857554$&$100$&$10^{10000}$&$0.4164824$&$1/64$&$0.0698280$&$2$&$10^{10000}$&$0.3224734$\\
$1/10$&$0.2856093$&$5$&$10^{10000}$&$0.4881779$&$1/64$&$0.4999542$&$100$&$10^{30}$&$0.4952642$\\
$1/10$&$0.5000258$&$100$&$10^{100}$&$0.4717227$&$1/64$&$0.4999542$&$100$&$10^{50}$&$0.4343990$\\
$1/10$&$0.2858613$&$100$&$10^{1000}$&$0.4092675$&$1/64$&$0.2856453$&$100$&$10^{100}$&$0.3694248$\\
$1/10$&$0.2856093$&$100$&$10^{10000}$&$0.3915972$&$1/64$&$0.0968954$&$100$&$10^{1000}$&$0.2451846$\\
$1/11$&$0.2857759$&$5$&$10^{1000}$&$0.4975185$&$1/64$&$0.0966197$&$100$&$10^{10000}$&$0.2060796$\\
$1/11$&$0.2856729$&$5$&$10^{10000}$&$0.4653739$&$1/128$&$0.0967824$&$2$&$10^{1000}$&$0.3836244$\\
$1/11$&$0.4999496$&$100$&$10^{100}$&$0.4612366$&$1/128$&$0.0555481$&$2$&$10^{10000}$&$0.2709904$\\
$1/11$&$0.2857759$&$100$&$10^{1000}$&$0.3901638$&$1/128$&$0.4999620$&$100$&$10^{30}$&$0.4874013$\\
$1/11$&$0.2856186$&$100$&$10^{10000}$&$0.3733105$&$1/128$&$0.4999620$&$100$&$10^{50}$&$0.4275029$\\
$1/12$&$0.2857673$&$5$&$10^{1000}$&$0.4788967$&$1/128$&$0.2856832$&$100$&$10^{100}$&$0.3590329$\\
$1/12$&$0.2857673$&$5$&$10^{10000}$&$0.4479777$&$1/128$&$0.0967824$&$100$&$10^{1000}$&$0.2236500$\\
$1/12$&$0.4999684$&$100$&$10^{100}$&$0.4528478$&$1/128$&$0.0555481$&$100$&$10^{10000}$&$0.1771132$\\
$1/12$&$0.2857673$&$100$&$10^{1000}$&$0.3755607$&$1/256$&$0.0315258$&$1$&$10^{10000}$&$0.4532725$\\
$1/12$&$0.2857673$&$100$&$10^{10000}$&$0.3593345$&$1/256$&$0.4999032$&$100$&$10^{30}$&$0.4835693$\\
$1/13$&$0.2856637$&$5$&$10^{1000}$&$0.4641766$&$1/256$&$0.4999032$&$100$&$10^{50}$&$0.4241374$\\
$1/13$&$0.2844701$&$5$&$10^{10000}$&$0.4341935$&$1/256$&$0.2857891$&$100$&$10^{100}$&$0.3540689$\\
$1/13$&$0.5000111$&$100$&$10^{50}$&$0.4973387$&$1/256$&$0.0968742$&$100$&$10^{1000}$&$0.2142779$\\
$1/13$&$0.5000111$&$100$&$10^{100}$&$0.4459882$&$1/256$&$0.0315258$&$100$&$10^{10000}$&$0.1603663$\\
$1/13$&$0.2857189$&$100$&$10^{1000}$&$0.3640172$&$1/512$&$0.0313678$&$1$&$10^{10000}$&$0.4232869$\\
$1/13$&$0.2856637$&$100$&$10^{10000}$&$0.3482876$&$1/512$&$0.4999329$&$100$&$10^{30}$&$0.4816696$\\
$1/14$&$0.2855528$&$5$&$10^{1000}$&$0.4522763$&$1/512$&$0.4999329$&$100$&$10^{50}$&$0.4224735$\\
$1/14$&$0.2422993$&$5$&$10^{10000}$&$0.4217659$&$1/512$&$0.2855712$&$100$&$10^{100}$&$0.3516266$\\
$1/14$&$0.4999153$&$100$&$10^{50}$&$0.4909741$&$1/512$&$0.0967377$&$100$&$10^{1000}$&$0.2098309$\\
$1/14$&$0.4999153$&$100$&$10^{100}$&$0.4402708$&$1/512$&$0.0313678$&$100$&$10^{10000}$&$0.1497831$\\
$1/14$&$0.2855528$&$100$&$10^{1000}$&$0.3547038$&$1/1024$&$0.0222964$&$1$&$10^{10000}$&$0.4076979$\\
$1/14$&$0.2855528$&$100$&$10^{10000}$&$0.3393624$&$1/1024$&$0.4998957$&$100$&$10^{30}$&$0.4807305$\\
$1/15$&$0.2814622$&$5$&$10^{1000}$&$0.4424253$&$1/1024$&$0.4998957$&$100$&$10^{50}$&$0.4216470$\\
$1/15$&$0.1984167$&$5$&$10^{10000}$&$0.4086917$&$1/1024$&$0.2856039$&$100$&$10^{100}$&$0.3504178$\\
$1/15$&$0.4999850$&$100$&$10^{50}$&$0.4855666$&$1/1024$&$0.0967583$&$100$&$10^{1000}$&$0.2076897$\\
$1/15$&$0.4999850$&$100$&$10^{100}$&$0.4354286$&$1/1024$&$0.0313244$&$100$&$10^{10000}$&$0.1449718$\\
$1/15$&$0.2857290$&$100$&$10^{1000}$&$0.3469773$&$1/2048$&$0.0177013$&$1$&$10^{10000}$&$0.3970552$\\
$1/15$&$0.2857290$&$100$&$10^{10000}$&$0.3319817$&$1/2048$&$0.4999310$&$100$&$10^{30}$&$0.4802571$\\
$1/16$&$0.2506745$&$5$&$10^{1000}$&$0.4333652$&$1/2048$&$0.4999310$&$100$&$10^{50}$&$0.4212344$\\
$1/16$&$0.1668332$&$5$&$10^{10000}$&$0.3938998$&$1/2048$&$0.2857073$&$100$&$10^{100}$&$0.3498109$\\
$1/16$&$0.4999745$&$100$&$10^{50}$&$0.4809436$&$1/2048$&$0.0965963$&$100$&$10^{1000}$&$0.2066454$\\
$1/16$&$0.4998122$&$100$&$10^{100}$&$0.4312874$&$1/2048$&$0.0177013$&$100$&$10^{10000}$&$0.1418188$\\
$1/16$&$0.2856104$&$100$&$10^{1000}$&$0.3405016$&&&&&\\
\hline
\end{tabular}
\end{table}

\begin{table}[p]
\centering\footnotesize
\setlength{\tabcolsep}{3pt}
\caption{Smallest $n$ with $C_2^J(\alpha,r,h_0,10^n)<1/2$, with $r$ chosen to minimise $C_2^J$. The last column is the resulting proportion $1-2C_2^J$ of zeros with real part in $[\alpha,1-\alpha]$, rounded down; since $C_2^J$ is decreasing in $t_0$, it is a lower bound for every $t_0\geq 10^n$. For each $\alpha$ the rows are for $h_0=100$ and, when it exists, for the smallest $h_0\in\{1,2,5,10\}$ with such an $n\leq 10^5$.}\label{tab:C2Jfrontier}
\begin{tabular}{|cc|c|cc||cc|c|cc|}
\hline
$\alpha$&$r$&$h_0$&$n$&$1-2C_2^J$&$\alpha$&$r$&$h_0$&$n$&$1-2C_2^J$\\
\hline
$1/7$&$0.5664448$&$100$&$5094$&$6.42\cdot10^{-7}$&$1/16$&$0.3863327$&$5$&$305$&$2.70\cdot10^{-4}$\\
$1/8$&$0.5699492$&$100$&$326$&$1.34\cdot10^{-4}$&$1/16$&$0.5756069$&$100$&$64$&$1.31\cdot10^{-3}$\\
$1/9$&$0.5035988$&$10$&$1157$&$1.75\cdot10^{-5}$&$1/32$&$0.2117772$&$2$&$838$&$3.64\cdot10^{-5}$\\
$1/9$&$0.5722963$&$100$&$178$&$8.89\cdot10^{-5}$&$1/32$&$0.5750673$&$100$&$44$&$2.58\cdot10^{-3}$\\
$1/10$&$0.5056153$&$10$&$428$&$9.33\cdot10^{-6}$&$1/64$&$0.2075002$&$2$&$447$&$5.18\cdot10^{-4}$\\
$1/10$&$0.5731932$&$100$&$129$&$5.77\cdot10^{-4}$&$1/64$&$0.5731529$&$100$&$38$&$4.73\cdot10^{-3}$\\
$1/11$&$0.3636817$&$5$&$2547$&$4.01\cdot10^{-6}$&$1/128$&$0.0599781$&$1$&$17023$&$4.07\cdot10^{-6}$\\
$1/11$&$0.5745233$&$100$&$104$&$1.82\cdot10^{-4}$&$1/128$&$0.5759244$&$100$&$35$&$1.06\cdot10^{-3}$\\
$1/12$&$0.3712620$&$5$&$921$&$4.01\cdot10^{-5}$&$1/256$&$0.0600973$&$1$&$6152$&$1.13\cdot10^{-5}$\\
$1/12$&$0.5742356$&$100$&$90$&$1.39\cdot10^{-3}$&$1/256$&$0.5745321$&$100$&$34$&$2.55\cdot10^{-3}$\\
$1/13$&$0.3769730$&$5$&$580$&$9.78\cdot10^{-5}$&$1/512$&$0.0599954$&$1$&$4585$&$4.19\cdot10^{-6}$\\
$1/13$&$0.5750566$&$100$&$80$&$9.64\cdot10^{-4}$&$1/512$&$0.5700385$&$100$&$34$&$7.91\cdot10^{-3}$\\
$1/14$&$0.3810698$&$5$&$435$&$1.79\cdot10^{-4}$&$1/1024$&$0.0598071$&$1$&$4055$&$5.36\cdot10^{-5}$\\
$1/14$&$0.5755687$&$100$&$73$&$7.86\cdot10^{-4}$&$1/1024$&$0.5756496$&$100$&$33$&$8.67\cdot10^{-4}$\\
$1/15$&$0.3841120$&$5$&$355$&$1.47\cdot10^{-4}$&$1/2048$&$0.0598000$&$1$&$3830$&$3.50\cdot10^{-5}$\\
$1/15$&$0.5754424$&$100$&$68$&$1.27\cdot10^{-3}$&$1/2048$&$0.5744293$&$100$&$33$&$2.20\cdot10^{-3}$\\
\hline
\end{tabular}
\end{table}

\begin{table}[p]
\centering\footnotesize
\setlength{\tabcolsep}{3pt}
\caption{Smallest $n$ with $C_2^L(\alpha,r,h_0,10^n)<1/2$, with $r$ chosen to minimise $C_2^L$. The last column is the resulting proportion $1-2C_2^L$ of zeros with real part in $[\alpha,1-\alpha]$, rounded down; since $C_2^L$ is decreasing in $t_0$, it is a lower bound for every $t_0\geq 10^n$. For each $\alpha$ the rows are for $h_0=100$ and, when it exists, for the smallest $h_0\in\{1,2,5,10\}$ with such an $n\leq 10^5$.}\label{tab:C2Lfrontier}
\begin{tabular}{|cc|c|cc||cc|c|cc|}
\hline
$\alpha$&$r$&$h_0$&$n$&$1-2C_2^L$&$\alpha$&$r$&$h_0$&$n$&$1-2C_2^L$\\
\hline
$1/7$&$0.5000553$&$100$&$189$&$1.43\cdot10^{-4}$&$1/16$&$0.2856104$&$5$&$287$&$3.67\cdot10^{-4}$\\
$1/8$&$0.3037963$&$10$&$17628$&$1.03\cdot10^{-7}$&$1/16$&$0.4999745$&$100$&$43$&$5.46\cdot10^{-3}$\\
$1/8$&$0.5000993$&$100$&$106$&$1.34\cdot10^{-4}$&$1/32$&$0.1665465$&$2$&$889$&$1.58\cdot10^{-4}$\\
$1/9$&$0.3235928$&$10$&$675$&$5.83\cdot10^{-5}$&$1/32$&$0.5000232$&$100$&$33$&$5.28\cdot10^{-3}$\\
$1/9$&$0.5000332$&$100$&$79$&$3.90\cdot10^{-5}$&$1/64$&$0.1453621$&$2$&$515$&$6.21\cdot10^{-5}$\\
$1/10$&$0.2856093$&$5$&$2545$&$1.14\cdot10^{-5}$&$1/64$&$0.4999542$&$100$&$30$&$9.47\cdot10^{-3}$\\
$1/10$&$0.5000258$&$100$&$66$&$7.79\cdot10^{-4}$&$1/128$&$0.0438441$&$1$&$14496$&$3.68\cdot10^{-6}$\\
$1/11$&$0.2857759$&$5$&$930$&$3.06\cdot10^{-5}$&$1/128$&$0.4999620$&$100$&$28$&$3.52\cdot10^{-3}$\\
$1/11$&$0.4999496$&$100$&$58$&$7.48\cdot10^{-4}$&$1/256$&$0.0353495$&$1$&$5572$&$1.70\cdot10^{-6}$\\
$1/12$&$0.2857673$&$5$&$593$&$8.86\cdot10^{-5}$&$1/256$&$0.4999032$&$100$&$28$&$1.13\cdot10^{-2}$\\
$1/12$&$0.4999684$&$100$&$53$&$1.89\cdot10^{-3}$&$1/512$&$0.0315405$&$1$&$4167$&$5.02\cdot10^{-5}$\\
$1/13$&$0.2856637$&$5$&$450$&$2.70\cdot10^{-4}$&$1/512$&$0.4999329$&$100$&$27$&$3.30\cdot10^{-3}$\\
$1/13$&$0.5000111$&$100$&$49$&$1.10\cdot10^{-3}$&$1/1024$&$0.0316041$&$1$&$3688$&$4.50\cdot10^{-5}$\\
$1/14$&$0.2858128$&$5$&$371$&$2.69\cdot10^{-4}$&$1/1024$&$0.4998957$&$100$&$27$&$5.25\cdot10^{-3}$\\
$1/14$&$0.4999153$&$100$&$46$&$2.69\cdot10^{-4}$&$1/2048$&$0.0315832$&$1$&$3486$&$2.57\cdot10^{-5}$\\
$1/15$&$0.2857290$&$5$&$321$&$2.62\cdot10^{-4}$&$1/2048$&$0.4999310$&$100$&$27$&$6.23\cdot10^{-3}$\\
$1/15$&$0.4999850$&$100$&$44$&$1.27\cdot10^{-3}$&&&&&\\
\hline
\end{tabular}
\end{table}

\appendix
\section{Bounds For Zeta}\label{sec:background}

In this Appendix we recall or establish bounds on $\zeta$ which we have used

First, we recall the subconvexity bounds 
\[  |\zeta(1/2+it)| \leq 0.618 |t|^{1/6} \log|t|\]
from \cite{HiaryPatelYang2024} with recent claimed improvements to
\[   |\zeta(1/2+it)| \leq 0.611 |t|^{1/6} \log|t|\]
in \cite{Revers2026}.
We note however that both works provide better asymptotic results, specifically \cite[Equation~(3.23)]{HiaryPatelYang2024} proves
\[|\zeta(1/2+it)| < 0.478013|t|^{\frac{1}{6}} \log|t| + 3.853165 |t|^{\frac{1}{6}} - 2.914229\]
for $t>10^{12}$ and \cite[Remark~3.1]{Revers2026}
\[0.470795|t|^{\frac{1}{6}}\log|t|+4.04972|t|^{\frac{1}{6}}-21.5437\]
for $t>7\cdot 10^{11}$, from which one can in fact conclude that
\[|\zeta(1/2+it)|  <  0.470795 |t|^{\frac{1}{6}}\log|t|   + 4.04972 |t|^{\frac{1}{6}} \]
for $|t| > 3$, which is better than $0.611 |t|^{1/6} \log|t|$ for large $t$ and worse for small $t$.

For even larger values of $t$ we have the better bound
\[  |\zeta(1/2+it)| \leq 66.7t^{\frac{27}{164}}\]
due to \cite{PatelYang2024}.

Inside the critical strip by \cite{Yang2024} we have that for any $k\geq 4$
\[ |\zeta(1-k/(2^k-2)+it)|< 1.546  |t|^{1/(2^k-2)}\log|t| \qquad t\geq 3.\]
A better bound for large $t$ and $\sigma$ close to $1$, valid for $|t|>3$ and $1/2 \leq \sigma \leq 1$,  is
\[ |\zeta(\sigma+it)|  < 70.7|t|^{4.438(1-\sigma)^{3/2}}(\log|t|)^{2/3} \]
 from \cite{Bellotti2024}.

Using the version of the  Phragm\'en-Lindel\"of from \cite{Fiori2026}, and following the approach of \cite[Section~4]{Fiori2026} we can essentially interpolate bounds on $\log|\zeta(\sigma+it)|$ as follows:
\begin{proposition}\label{prop:interpolation}
For $t>12$
 we have the following bounds on $\log|\zeta(\sigma+it)|$ .
\begin{enumerate}
    \item For $1/2 \leq \sigma \leq 5/7$, we have the following bounds
    \begin{align*}
     \log|\zeta(\sigma+it)|
     &\leq \left(\frac{47}{123} - \frac{107}{246}\sigma\right)\log|t|+ \frac{14}{3}\left(\sigma-\frac{1}{2}\right)\log\log|t|\\&\quad  + \frac{14}{3}\left(\frac{5}{7}-\sigma\right)\log(66.7) + \frac{14}{3}\left(\sigma-\frac{1}{2}\right)\log(1.546)\\&\quad + \Ostar\!\left(\frac{3}{2t^2} + \frac{283}{4t^2\log|t|}\right)\\
     \log|\zeta(\sigma+it)| &\leq  \left(\frac{4(1-\sigma)}{9}-\frac{1}{18} \right)\log|t|
                 +\log\log|t| +   \frac{14}{3}\left(\frac{5}{7}-\sigma\right)\log(0.611) \\&\quad+  \frac{14}{3}\left(\sigma-\frac{1}{2}\right)\log(1.546)   \\&\quad+ \Ostar\!\left(\frac{3}{2t^2} + \frac{579}{4t^2\log|t|}\right) \\
     \log|\zeta(\sigma+it)| &\leq  \left(\frac{4(1-\sigma)}{9}-\frac{1}{18} \right)\log|t|
                 +\log\log|t| \\&\quad +  \frac{14}{3}\left(\frac{5}{7}-\sigma\right)\log\left(0.470795  + \frac{4.04972}{\log|t|}\right) +\frac{14}{3}\left(\sigma-\frac{1}{2}\right)\log(1.546)\\&\quad   +  \Ostar\!\left(\frac{3}{2t^2} + \frac{567}{4t^2\log|t|}\right)
     \end{align*}
    \item For  $k\geq 4$ and $1-\frac{k}{2^k-2} \leq \sigma \leq 1-\frac{k+1}{2^{k+1}-2}$ (this line only needs $t>3$)
        \begin{align*} \log|\zeta(\sigma+it)| &\leq  \left(\frac{(1-\sigma)}{k-1+2^{1-k}}-\frac{1}{2^k(k-1)+2} \right)\log|t|  +\log\log|t| + \log(1.546)\\&\quad +\Ostar\!\left(\frac{3}{2t^2} + \frac{7}{t^2\log|t|}\right)
       \end{align*}
\end{enumerate}
\end{proposition}
\begin{proof}
    Applying the Phragm\'en-Lindel\"of Theorem as in \cite[Section~4]{Fiori2026} to bound $|\zeta|$ and taking the logarithm will allow one to obtain the result.

    The hardest case which is new is perhaps the third bound in (1) so we will illustrate how the proof goes for this case.
    We define
    \[ G_0(s) = (1+s), \]
    \[ G_1(s) = \frac{1}{2}(\log(4e + 1+s) + \log(4e +1 - s))0.470795  +  4.04972, \]
    and
    \[ G_2(s) = \frac{1.546}{2}(\log(4e + s) + \log(4e +2 - s)). \]
    The bounds we wish to interpolate are:
    \[ |\zeta(1/2 + it)| < |G_0(1/2+it)|^{1/6}|G_1(1/2+it)|^{1}|G_2(1/2+it)|^{0}\]
    and
    \[ |\zeta(5/7 + it)| < |G_0(5/7+it)|^{1/14}|G_1(5/7+it)|^{0}|G_2(5/7+it)|^{1}.\]
Note that the functions $G_0$ and $G_1$ are increasing in $\sigma$ for $|t|$ large whereas $G_2$ is decreasing, they thus satisfy the conditions coming from \cite[Theorem~7]{Fiori2026} related to how the exponents are changing.
One can also verify that the necessary condition is satisfied in the compact region $|t|\leq 4e+1\approx 11.87$ 
via the Maximum Modulus Principle.
The final $\Ostar$ term is established by looking at how quickly $G_0$, $G_1$, and $G_2$ converge to $|t|$, $0.470795\log|t|+4.04972$, and $1.546\log|t|$.
\end{proof}

To get bounds for $\zeta(\sigma + it)$ for $-1 \leq \sigma \leq 1/2$ one can use the following.

\begin{lemma}\label{lem:zetalessthanhalf}
For $0 \leq \sigma \leq 1/2$ and $t>10$ we have
\begin{align*} \log|\zeta(\sigma+it)| = \log|\zeta(1-\sigma+it)| &+ (1/2-\sigma)\log|t/2|  +(\sigma-1/2)\log|\pi|
\\&+ \Ostar\!\left(\left(\frac{\sqrt3}{216}+\frac{0.17}{t}\right)\frac{1}{t^2}\right). \end{align*}
For $-1\leq \sigma \leq 0$ and $t>10$ we have:
\begin{align*} \log|\zeta(\sigma+it)| = \log|\zeta(1-\sigma+it)| &+ (1/2-\sigma)\log|t/2|  +(\sigma-1/2)\log|\pi|
\\&+  \Ostar\!\left(\left(\frac{1}{2}+\frac{1.4}{t}\right)\frac{1}{t^2}\right). \end{align*}
\end{lemma}
\begin{proof}
The functional equation for $\zeta$ gives
\[  \log|\zeta(\sigma+it)| = \log|\zeta(1-\sigma+it)| +(\sigma-1/2)\log|\pi| + \log\left| \frac{\Gamma((1-\sigma+it)/2)}\Gamma((\sigma+it)/2)\right|.\]
A precise  version of Stirling's approximation (See \cite[Section~2]{Brent2019}) gives us
\[ \log\Gamma(z) = (z-1/2)\log z - z + \frac{1}{2}\log(2\pi) + \frac{B_{2}}{2z} + R_2(z) \]
for $\Re(z)>0$ where $|R_2(z)| < |\frac{(1+\sqrt{2\pi})B_{4}}{12z^3}|$.
From this we may derive that for $0\leq a,b \leq 1$ and $t>2$:
\begin{align*}
 \log(\Gamma(& a+it/2))- \log(\Gamma(b+it/2)) \\
 &= (a+it/2-1/2)\log(a+it/2) - (b+it/2-1/2)\log(b+it/2) + (b-a) \\&\quad+ \frac{B_2(a-it/2)}{2(a^2+t^2/4)} - \frac{B_2(b-it/2)}{2(b^2+t^2/4)} + R_2(a+it/2) - R_2(b+it/2)\\
 &= (a-b)\log(t/2) + (a+it/2-1/2)\log(i+2a/t)
 \\&\quad - (b+it/2-1/2)\log(i+2b/t) + (b-a)
 \\&\quad+ \frac{B_2(a-it/2)}{2(a^2+t^2/4)} - \frac{B_2(b-it/2)}{2(b^2+t^2/4)} + R_2(a+it/2) - R_2(b+it/2).
\end{align*}
Taking real parts we obtain
\begin{align*}
 \log|\Gamma(& a+it/2)| - \log|\Gamma(b+it/2)| \\
 &= (a-b)\log(t/2) + \frac{a-1/2}{2}\log(1+4a^2/t^2)  - \frac{b-1/2}{2}\log(1+4b^2/t^2)\\&\quad + \frac{t}{2}\left(\arctan\left(2a/t\right) - \arctan\left(2b/t\right)\right) + (b-a)
  \\&\quad + \frac{B_2 a }{2(a^2+t^2/4)}-\frac{B_2 b }{2(b^2+t^2/4)}
  \\&\quad + \Re( R_2(a+it/2) - R_2(b+it/2) )\\
&=(a-b)\log\left(t/2\right) +
\frac{2a^2(a-1/2) - 2b^2(b-1/2) - \frac{4}{3}(a^3-b^3)}{t^2}
\\&\quad + \frac{2B_2}{t^2}\frac{(a-b)(1-4ab/t^2)}{(1 + 4a^2/t^2)(1 + 4b^2/t^2)  }
\\&\quad + \Ostar\!\left(  \frac{4|a-1/2|a^4 + 4|b-1/2|b^4}{t^4} \right) + \Ostar\!\left(\frac{16(a^5+b^5)}{5t^4}\right)
\\&\quad + \Ostar\!\left( \frac{4B_4(1+\sqrt{2\pi})}{3t^3} \right)\\
 \end{align*}
A more careful analysis of the error terms yields the result.

 For $-1\leq \sigma \leq 0$ we must apply
\begin{align*}
\log|\Gamma((\sigma+it)/2)| &= \log|\Gamma((\sigma+it)/2+1)| - \log|(\sigma+it)/2)| \\
&=  \log|\Gamma((\sigma+it)/2+1)| - \log|t/2| - \frac{1}{2}\log(1 + \sigma^2/t^2) \\
&= \log|\Gamma((\sigma+it)/2+1)|  - \log|t/2| + \Ostar\!\left(\frac{1}{t^2}\right).
\end{align*} 
Additionally, in the comparison
 \[  \log|\Gamma(1+\sigma/2+it/2)| - \log|\Gamma(1/2-\sigma/2+it/2)| \]
we now take $a=1/2-\sigma/2$ and $b=\sigma/2+1$ will have $a-b=-1/2-\sigma$ and $1/2\leq a,b\leq 1$. As above, combining terms and checking bounds will give the result.
\end{proof}

A similar argument also bounds changes in the argument of $\Gamma$:
\begin{lemma}\label{lemma:argGamma}
For $-1/2 \leq \delta \leq 1$ and $t>10$ we have
\[ \arg\Gamma(1/2+it/2) - \arg\Gamma(1/2+\delta/2+it/2) = -\frac{\delta\pi}{4} + \frac{\delta^2}{4t} + \Ostar\!\left(\frac{0.21}{t^3}\right)\]
Note that here $\arg\Gamma$ refers to the analytic continuation, not the principal branch.
\end{lemma}
\begin{proof}
We start from 
\begin{align*}
 \log(\Gamma(& a+it/2))- \log(\Gamma(b+it/2)) \\
 &= (a-b)\log(t/2) + (a+it/2-1/2)\log(i+2a/t)
 \\&\quad - (b+it/2-1/2)\log(i+2b/t) + (b-a)
 \\&\quad+ \frac{B_2(a-it/2)}{2(a^2+t^2/4)} - \frac{B_2(b-it/2)}{2(b^2+t^2/4)} + R_2(a+it/2) - R_2(b+it/2).
\end{align*}
Now taking the imaginary part gives
\begin{align*}
 \arg\Gamma(& a+it/2)- \arg\Gamma(b+it/2) \\
  &= \frac{t}{4}\log\left|\frac{1+4a^2/t^2}{1+4b^2/t^2} \right| + (a-b)\frac{\pi}{2} - (a-1/2)\arctan(2a/t)
  \\&\quad + (b-1/2)\arctan(2b/t)
  \\&\quad+ \frac{4B_2}{t^3} \frac{a^2 - b^2}{(1+4a^2/t^2)(1+4b^2/t^2)} + \Im( R_2(a+it/2) - R_2(b+it/2) ).
\end{align*}
Taking $a=1/2$ and $b=1/2+\delta/2$, then several applications of Taylor's Theorem to verify upper and lower bounds, gives the result.
\end{proof}

\begin{lemma}\label{lemma:logderzetabound}
Write the Laurent series for $\zeta'/\zeta$ centered at $1$ as
\[ -\frac{\zeta'}{\zeta}(1+z) = \frac{1}{z} + \sum_{n=0}^{\infty}  (-1)^{n+1} a_n z^n \]
then we have for $n\geq 1$
\[ a_n = 3^{-n-1} +5^{-n-1} + 7^{-n-1}+ \Ostar\!\left(\frac{6}{10} \cdot 7^{-n}\right) \]
and hence we have for any $x\in (0,2]$ and any $N\geq 0$ that
\[  \frac{1}{x} +  \sum_{n=0}^{2N}  (-1)^{n+1} a_n x^n < -\frac{\zeta'}{\zeta}(1+x) <  \frac{1}{x} +  \sum_{n=0}^{2N-1}  (-1)^{n+1} a_n x^n \]
and hence
\[ -\log(x) +  \sum_{n=0}^{2N-1}  (-1)^n \frac{a_n}{n+1} x^{n+1} < \log(\zeta(1+x)) < -\log(x)  +  \sum_{n=0}^{2N} (-1)^n \frac{a_n}{n+1} x^{n+1} \]
where an upper index of $-1$ denotes the empty sum.
\end{lemma}
\begin{proof}
The first claim follows the proof of Chirre-Helfgott \cite[Lemma~A.7]{ChirreHelfgott2025}, their proof there attributed to
A. Kalmynin \cite{Kalmynin2025}. We instead use a radius of $7$ and remove the singularities at $1$, $-2$, $-4$, and $-6$.
Setting \[ G(z) = \frac{\zeta'}{\zeta}(1+z) + \frac{1}{z} - \frac{1}{z+3} - \frac{1}{z+5} - \frac{1}{z+7}. \]
We bound the Taylor coefficients of $G(z)$, which is analytic on $|z|\leq 7$ by the Cauchy integral formula.
We rigorously verify the bound on the boundary of $|G(z)| < \frac{6}{10}$ for $|z|=7$. Rearranging and elementary calculations then give the result.

The second claim is then a simple consequence of this being an alternating series.
We see that the sequence is positive by noting that $a_0$ is positive and that the only negative term decreases at least as quickly as the positive terms.
The same idea applied to $\frac{d}{dn} (a_n2^n)$ gives that the $a_n2^n$ are a decreasing sequence for $n\geq 2$, one can checks the first few terms through direct computations.

The third comes from the first using
\[ \log(\zeta(1+x)) = -\log(x) + \int_0^{x} \frac{d}{du}(\log(u\zeta(1+u)))\,du  =   -\log(x) + \int_0^{x} \left(\frac{1}{u} + \frac{\zeta'}{\zeta}(1+u)\right)\,du\]
\end{proof}
\begin{remark}\label{lemma:logderzetabound-remark}
The coefficients $a_n$ above have the formulas
\[ a_n =  -\frac{1}{n!}\lim_{m\rightarrow \infty} \left( \left( \sum_{k=1}^{m}  \frac{\log(k)^n\Lambda(k)}{k} \right) - \frac{(\log m)^{n+1}}{n+1}\right) .\]
and are related to the Stieltjes constants (See \cite{Coffey2009}).
\end{remark}


\end{document}